\documentclass[11pt, a4paper]{article}
\usepackage{fullpage}

\newif\ifabstr
\abstrfalse

\usepackage[latin1]{inputenc}

\usepackage{amsmath}
\usepackage{amssymb,amsthm,graphicx}
\usepackage[dvips]{epsfig}
\usepackage{amsfonts}
\usepackage{xcolor}
\usepackage{hyperref}
\usepackage{enumerate}

\newtheorem{theorem}{Theorem}
\newtheorem{proposition}[theorem]{Proposition}
\newtheorem{corollary}[theorem]{Corollary}
\newtheorem{observation}[theorem]{Observation}

\newtheorem{lemma}[theorem]{Lemma}

\newtheorem{claim}{Claim}[theorem]
\theoremstyle{remark}

\newtheorem{remarks}[theorem]{Remarks}

\theoremstyle{definition}
\newtheorem{definition}[theorem]{Definition}

\newcommand{\Z}{\mathbb{Z}}

\newcommand{\G}{{\mathbf G}}

\newcommand{\K}{{\mathbf K}}
\renewcommand{\L}{{\mathbf{L}}}

\renewcommand{\O}{{\mathbf{O}}}

\newcommand{\X}{\mathbf{X}}
\newcommand{\Y}{\mathbf{Y}}

\definecolor{darkgreen}{rgb}{0,.5,0}

\newcounter{sideremark}
\newcommand{\marrow}{\stepcounter{sideremark}\marginpar{$\boldsymbol{\longleftarrow\scriptstyle\arabic{sideremark}}$}}

\newcommand{\michael}[1]{{\color{red}\vskip 5pt\textsf{*** (Michael) \marrow #1\vskip 5pt}}}

\DeclareMathOperator{\lk}{lk}
\DeclareMathOperator{\sd}{sd}
\DeclareMathOperator{\st}{st}

\title{Shellings and sheddings induced by collapses\thanks{This work was
supported by the grant no.~19-04113Y of the Czech Science Foundation
(GA\v{C}R). 
T.M. is supported by grant ANR-17-CE40-0033 of the French National Research
Agency ANR (SoS project).
M.T. is partially supported by Charles University project
UNCE/SCI/004. 
}} 

\author{
Thomas Magnard
\\ \small{Universit\'e Paris-Est, LIGM, CNRS, ENPC, ESIEE Paris, UPEM, Marne-la-Vall\'ee, France}
\and
Michael Skotnica
\\ \small{Department of Applied Mathematics, Charles University, Prague, Czech
Republic}
\and Martin Tancer
\\ \small{Department of Applied Mathematics, Charles University, Prague, Czech
Republic}}

\date{}

\begin{document}

\maketitle



\begin{abstract}
We say that a pure simplicial complex ${\mathbf K}$ of dimension $d$ satisfies the removal-collapsibility condition if ${\mathbf K}$ is either empty or ${\mathbf K}$ becomes collapsible after removing $\tilde \beta_d ({\mathbf K}; {\mathbb Z}_2)$ facets, where $\tilde \beta_d ({\mathbf K}; {\mathbb Z}_2)$ denotes the $d$th reduced Betti number.

In this paper, we show that if the link of each face of a pure simplicial complex ${\mathbf K}$ (including the link of the empty face which is the whole ${\mathbf K}$) satisfy the removal-collapsibility condition, then the second barycentric subdivision of ${\mathbf K}$ is vertex decomposable and in particular shellable. This is a higher dimensional generalization of a result of Hachimori, who proved that if the link of each vertex of a pure 2-dimensional simplicial complex ${\mathbf K}$ is connected, and ${\mathbf K}$ becomes simplicially collapsible after removing $\tilde{\chi}({\mathbf K})$ facets, where $\tilde \chi ({\mathbf K})$ denotes the reduced Euler characteristic, then the second barycentric subdivision of ${\mathbf K}$ is shellable.

For the proof, we introduce a new variant of decomposability of a simplicial complex, stronger than vertex decomposability, which we call \emph{star decomposability}. This notion may be of independent interest.

\end{abstract}


\section{Introduction}
Shellability and collapsibility (to be defined later) are two widely used approaches for
combinatorial decomposition of a simplicial complex. They are similar in
spirit, yet there are important differences among those two notions.
There are shellable complexes homotopy equivalent to a wedge of spheres,
whereas no non-trivial wedge can be collapsible. On the other hand, two
triangles sharing a vertex provide an example of a collapsible complex that
is not shellable. Yet in some important cases, one can relate these two
notions. 

The easy direction is that shellability implies collapsibility whenever the
complex is contractible (in fact, whenever the complex has trivial homology). We will focus here on a more interesting
direction: when collapsibility implies shellability?

In this spirit, Hachimori~\cite{hachimori08} proved that for a pure $2$-dimensional simplicial
complex $\K$, the following statements are equivalent: 
\begin{enumerate}[(i)]
  \item The complex $\K$ has a shellable subdivision.
  \item The second barycentric subdivision $\sd^2 \K$ is shellable.
  \item The link of each vertex of $\K$ is connected and $\K$ becomes
    collapsible after removing $\tilde \chi (\K)$ facets where $\tilde \chi$
    denotes the reduced Euler characteristic.
\end{enumerate}



As Hachimori points out, one cannot expect that such an equivalence would be
achievable in higher dimensions. Namely, the implication (i)~$\Rightarrow$~(ii)
cannot hold in higher dimensions due to the examples by
Lickorish~\cite{lickorish91}.
However, we will show that it is possible to
generalize the interesting implication (iii)~$\Rightarrow$~(ii). The
equivalence of (iii) and (ii) was one of the important steps in a recent proof
of NP-hardness of recognition of shellable
complexes~\cite{goaoc-patak-patakova-tancer-wagner19}. Though the hardness
reduction requires the implication only in dimension 2, we find it interesting
to provide a higher-dimensional generalization. For example, the computational
complexity status of recognition of shellable/collapsible 3-spheres is unknown
and the implication (iii)~$\Rightarrow$~(ii) could provide a link between the
two notions.


For explaining our generalization, we briefly introduce some notions (see also
Section~\ref{s:prelim} for the notions undefined here). 

\paragraph{Collapsibility.}
Let $\K$ be a simplicial complex
and $\sigma \in \K$ be a face which is contained in only one face $\tau
\in K$ with $\sigma \subsetneq \tau$. (Necessarily $\dim \tau = \dim \sigma +
1$ and $\tau$ is a \emph{facet} $\K$, that is, an inclusion-wise maximal face of $\K$). In this case,
we say that $\sigma$ is a \emph{free} face of $\K$ and we also say that a
complex $\K'$ \emph{arises from $\K$ by an elementary collapse} if there are
$\sigma$ and $\tau$ as above such that $\K' = \K \setminus \{\sigma, \tau\}$,
we denote this by $\K \searrow \K'$. A complex $\K$ is \emph{collapsible}, if
there is a sequence $(\K_1, \dots, \K_r)$ of complexes such that $\K_1 = \K$,
$\K_r$ is a point, and $\K_1 \searrow \K_2 \searrow \cdots \searrow \K_r$. An
important property of collapsibility is that the elementary collapses preserve
the homotopy type, a fortiori, the homology groups.


\paragraph{Shellability.}
Let $\K$ be a simplicial complex of dimension $k$. 
A total order $F_1, \ldots, F_t$ of facets of $\K$ 
is called a \emph{shelling} if 
$F_i \cap \bigcup_{j=1}^{i-1} F_j$
is a pure
$(k-1)$-dimensional complex. (Purely formally, we consider the facets in
the formula $F_i \cap \bigcup_{j=1}^{i-1} F_j$ above as subcomplexes of $\K$.) 
$\K$ is then said to be \emph{shellable} if it admits a shelling order.
For comparison with collapsibility, we will also use the \emph{reverse
shelling order} $F_t, \ldots, F_1$.


\paragraph{Removal-collapsibility condition.}
We will say that a pure complex $\K$ satisfies the
\emph{removal-collapsibility} condition, abbreviated to (RC) condition, if $\K$
is either empty or $\K$
becomes collapsible after removing some number of facets. We remark that if
$\dim \K = d$ the number of removed facets can be easily computed as $\tilde
\beta_d (\K; \Z_2)$ where $\tilde \beta_d (\K; \Z_2)$ denotes the reduced $d$th Betti number,
i.e., the rank of the reduced homology group $\tilde H_d(\K;
\Z_2)$.\footnote{The choice of coefficients $\Z_2$ is not very important here.
We could choose an arbitrary field.} Indeed, by a routine application of the
Mayer-Vietoris exact sequence, removing a facet either decreases $\tilde
\beta_d (\K; \Z_2)$ by one or increases $\tilde \beta_{d-1} (\K; \Z_2)$ by one. But we cannot
afford the latter case if the complex becomes collapsible after removing some
number of facets. In addition, the lower dimensional homology remains
unaffected when removing a facet (directly from the definition of simplicial
homology or again by a Mayer-Vietoris exact sequence), therefore a complex
satisfying (RC) condition also satisfies $\tilde \beta_i (\K; \Z_2) = 0$ for
$0 \leq i \leq d-1$. In particular, $\tilde \chi(K) = (-1)^d \tilde \beta_d(\K;
\Z_2)$.

We also observe that if $d = 1$, that is, if $\K$ is a graph, then the (RC) condition is equivalent with
stating that $\K$ is connected. Also, every $0$-complex satisfies the (RC)
condition.

Altogether, Hachimori's condition (iii) for $2$-complexes is equivalent to
saying that the link of the empty face (i. e., $\K$) and the link of every
vertex satisfies the (RC) condition. This is furthermore equivalent with saying
that the link of every face of $\K$ satisfies the (RC) condition as links of
dimension at most $0$ always satisfy the (RC) condition. We say that $\K$
satisfies the \emph{hereditary removal-collapsibility} condition, abbreviated
to (HRC) condition, if the link of every face of $\K$ satisfies the (RC)
condition. In particular, (HRC) is equivalent to Hachimori's condition (iii) for
$2$-complexes. This condition is hereditary in the following sense: If $\K$
satisfies (HRC) and $\sigma \in \K$, then the link $\lk (\sigma, \K)$ also
satisfies (HRC). Indeed, the link of $\sigma'$ in $\lk (\sigma, \K)$ is just
the 
link of $\sigma \cup \sigma'$ in $\K$.\footnote{Note that we do not claim that
(HRC) is hereditary with respect to subcomplexes or induced subcomplexes.} 

We establish the following generalization of Hachimori's implication (iii)
$\Rightarrow$ (ii).

\begin{theorem}
\label{t:main}
Let $\K$ be a pure simplicial $d$-complex satisfying the (HRC) condition, then the second barycentric subdivision $\sd^2 \K$ is
shellable. 
\end{theorem}

We suspect that the reverse implication does not hold but we are not
aware of a concrete complex violating the reverse implication. Possibly
interesting examples could be the non-collapsible triangulations of the $3$-ball
$B_{15,66}$ and $B_{17,95}$ constructed by Benedetti and
Lutz~\cite{benedetti-lutz13} but we do not know if their second barycentric
subdivisions are shellable.

For the proof of Theorem~\ref{t:main}, we will define two coarser notions than shellability called
\emph{star decomposability} and \emph{star decomposability in vertices}, which
may be of independent interest. Together with \emph{vertex decomposability} of
Provan and Billera~\cite{provan-billera80} we will establish the following
chain of implications, where the last implication is a result of Provan and
Billera. 

\medskip

star decomposable in vertices $\Rightarrow$ star decomposable $\Rightarrow$
vertex decomposable $\Rightarrow$ shellable

\medskip

Therefore, for a proof of Theorem~\ref{t:main} it is sufficient to prove the 
following generalization (together with the first two promised implications).

\begin{theorem}
\label{t:star_decomposable}
Let $\K$ be a pure simplicial $d$-complex satisfying the (HRC) condition, then the second barycentric subdivision $\sd^2 \K$ is star
decomposable in vertices. 
\end{theorem}

\paragraph{Additional motivation and background.}
Both notions, collapsibility and shellability, play an important role in
PL topology because they may help to determine not only the homotopy type of a given
collapsible/shellable space but sometimes even the (PL) homeomorphism type. For
example, a collapsible PL manifold is a ball, and a shelling of a PL-manifold
(if it does not change the homotopy type) preserves the homeomorphism
type~\cite{rourke-sanderson82}.

A relation between collapsibility or shellability of some subdivision of a
complex and of some barycentric subdivision has been studied by Adiprasito and
Benedetti~\cite{adiprasito_benedetti17}. Namely, they show that a simplicial
complex is PL homeomorphic to a shellable complex if and only if it is
shellable after finitely many barycentric subdivisions,\footnote{The result is
stated in terms of derived subdivisions but there is no difference on
the combinatorial level.} and they show an analogous result for collapsibility. If
we were interested only in shellability of some barycentric subdivision of $\K$
in Theorem~\ref{t:main}, it is possible that the proof could be easier, because
it would be possible to use arbitrary suitable subdivisions in the intermediate
steps. 


Hachimori's implication (iii)$\Rightarrow$(ii), as well as its generalization,
Theorem~\ref{t:main}, can be understood as a tool for showing that a concrete
complex is shellable. A lot of effort has been devoted to developing 
such tools in various contexts; see e.g.~\cite{bjorner-wachs83,kozlov97}. The
advantage of Theorem~\ref{t:main} could be that the (HRC) condition may
naturally follow from the topological/combinatorial properties of a considered problem as
it is in the case of the application of Hachimori's result
in~\cite{goaoc-patak-patakova-tancer-wagner19}. A possible disadvantage could
be that we have to allow some flexibility on the target complex (it has to be
the second barycentric subdivision of another complex).

An additional piece of motivation may come from commutative algebra. 
For
example, Herzog and Takayama~\cite{herzog-takayama02} found out that if $\K$
is a complex (not necessarily pure) and $I_\K$ is the Stanley-Reisner ideal
corresponding to $\K$, then $I_\K$ has linear quotients if and only the
Alexander dual $\K^*$ is shellable (in the non-pure sense, but the pure case is a
special case, of course). Thus, Theorem~\ref{t:main} may serve as a
tool showing that certain Stanley-Reisner ideals have linear quotients.

Finally, the notions of star decomposability and star decomposability in
vertices that we introduce along the way may be of independent interest as
inductive tools similar to collapsibility, shellability,
vertex-decomposability, etc. Although their definitions are slightly technical,
they appear very naturally in our context, as we sketch in the proof strategy below.
It would also be interesting to know whether these notions admit some counterpart
in terms of commutative algebra (similarly to the Herzog-Takayama equivalence
above).

\paragraph{Proof strategy.}
Here we first sketch Hachimori's proof (iii) $\Rightarrow$ (ii), in our words
though. Then we sketch the necessary steps for upgrading the proof to higher
dimensions. 

Let $\K$ be a pure $2$-complex satisfying the conditions of (iii). We want to
sketch a strategy how to shell $\sd^2 \K$. For simplicity of pictures, we will
assume that $\K$ is already collapsible (as we want to avoid the non-trivial
second homology in the pictures).

The second barycentric subdivision $\sd^2 \K$ is covered by stars of vertices
of $\sd^2 \K$ which correspond to original faces of $\K$; see
Figure~\ref{f:star_cover}. The stars may overlap, but they overlap only in
their boundaries (in links). Now, let us consider an elementary collapse $\K
\searrow \K'$ while removing a free face $\sigma$ and a maximal face $\tau$
containing $\sigma$. Naturally, in $\sd^2 \K$ we want to emulate this by a reverse
shelling removing the triangles first in $\st(\sigma, \sd^2 \K)$ and then in
$\st(\tau, \sd^2 \K)$;\footnote{Formally speaking,
$\st(\sigma, \sd^2 \K)$ stands for $\st(\{\{\sigma\}\}, \sd^2 \K)$, etc.; see our
convention in the preliminaries.} see Figure~\ref{f:star_cover_shell}.
This is indeed a good strategy as Hachimori~\cite{hachimori08}
showed. However, this quite heavily depends on the fact that the dimension of the
complex is $2$ as the structure of $\sd^2 \K$ is so simple that all steps are
obvious. 

\begin{figure}
\begin{center}
  \includegraphics[page=1]{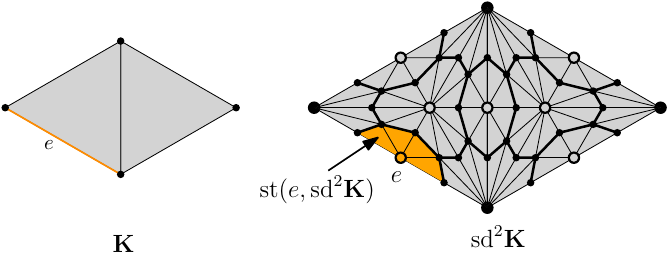}
\end{center}
\caption{Decomposition of $\sd^2 \K$ into stars. For example, an edge $e$ of
$\K$ becomes a vertex in $\sd^2 \K$. Consequently, its star in $\sd^2 \K$ is
one of the stars in the decomposition.}
\label{f:star_cover}
\end{figure}

\begin{figure}
\begin{center}
  \includegraphics[page=2]{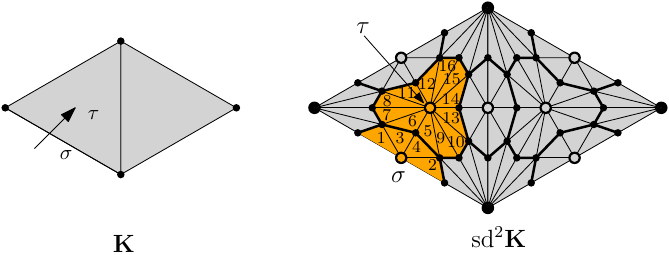}
\end{center}
\caption{Reverse shelling of $\sd^2 \K$ following an elementary collapse of
  $\K$. The numbers in triangles indicate a valid order of removing triangles.}
\label{f:star_cover_shell}
\end{figure}

In general dimension we want to proceed similarly. However it
seems out of reach to describe directly the order of removals of facets of
$\sd^2 \K$ and check that this is a shelling order due to a complicated
structure of $\sd^2 \K$. At least we initially tried this approach but we
quickly got lost in addressing too many cases. Therefore, we instead use
the aid of some coarser notions.


The first helpful notion is vertex decomposability of Provan and
Billera~\cite{provan-billera80}. 
A simplicial $d$-complex $\K$ is \emph{vertex decomposable} if it is pure and
\begin{itemize}
  \item $\K$ is a $d$-simplex, or
  \item there is a vertex $v \in V(\K)$ such that $\K - v$ is
    $d$-dimensional vertex decomposable (where $\K - v$ denotes the complex
    obtained by removing $v$ and all faces containing $v$ from $\K$)
    and $\lk(v, \K)$ is $(d-1)$-dimensional vertex decomposable.
\end{itemize}
This recursive definition induces an order $v_1, \ldots, v_{n-(d+1)}$ of
$n - (d+1)$ vertices of $\K$ according to the sequence of vertex removals in
the second item (where $n$ is the number of vertices of $\K$). This order is called a \emph{shedding order} and we
artificially extend any shedding order to all vertices of $\K$ so that the
remaining vertices follow in arbitrary order. (Intuitively, as soon as we reach
a $d$-simplex in the first item, we allow removing vertices in arbitrary order.)

Proving that $\sd^2 \K$ is vertex
decomposable is stronger than showing that $\sd^2 \K$ is shellable, and it also seems easier to specify the shedding order
as we deal with a smaller number of objects. For example, in case of the collapse from
Figure~\ref{f:star_cover_shell}, we specify the order only on three vertices;
see Figure~\ref{f:star_cover_vertex}.

\begin{figure}
\begin{center}
  \includegraphics[page=3]{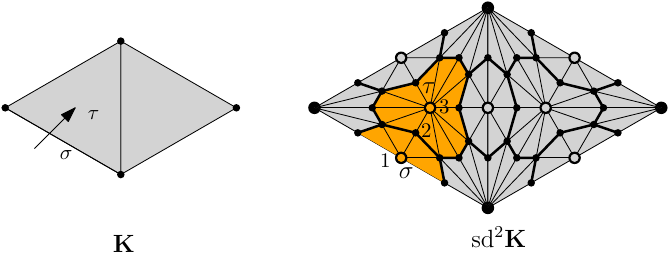}
\end{center}
\caption{Vertex decomposition (shedding) of $\sd^2 \K$ following an elementary collapse of
  $\K$. In this case, we first remove $\sigma$, then the vertex in between of
$\sigma$ and $\tau$ and finally $\tau$.}
\label{f:star_cover_vertex}
\end{figure}

On the other hand, it is even easier to start removing the closed stars of vertices
(and then taking a closure to get again a simplicial complex). In case of
Figure~\ref{f:star_cover_vertex}, we would first remove the closed star of
$\sigma$ in $\sd^2 \K$. Subsequently, when taking the closure, we reintroduce the full link of
$\sigma$. Thus in this case, our first step coincides with removing $\sigma$
(and therefore the open star of $\sigma$). The second step is, however, more
interesting (see Figure~\ref{f:star_cover_overlap}): First we
remove the closed star of $\tau$. Then, when taking the closure, we
do not reintroduce the vertex in between of $\sigma$ and $\tau$.
Therefore, this second step removes simultaneously two vertices.

\begin{figure}
\begin{center}
  \includegraphics[page=4]{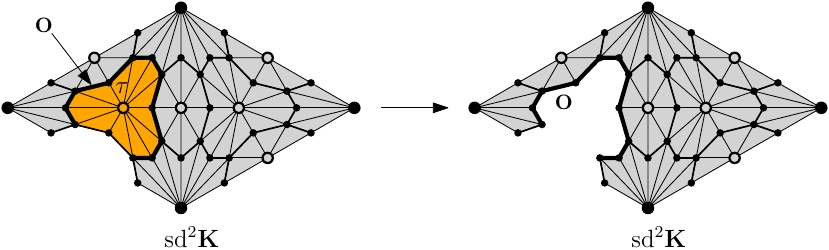}
\end{center}
\caption{Overlap of the link of $\tau$ and the rest of the complex.}
\label{f:star_cover_overlap}
\end{figure}

This will be our notion of star decomposability; however, one of the key steps
in our approach is to identify an appropriate property of order of removals as above,
which implies vertex decomposability of our complex. 
For sketching the idea, let us again consider the case of removing the closed star of 
$\tau$ in the second step above. Similarly as in the case of vertex decomposability, we will need
that the link of the center of the removed star (in this case the link of
$\tau$) is star decomposable. However, this is not the only condition that we
require. Let $\O$ be the overlap of the link of $\tau$ and the remainder of the
complex after removing the star of $\tau$ (see
Figure~\ref{f:star_cover_overlap}). We will actually need a star decomposition
of the link of $\tau$ such that $\O$ is an intermediate step in this
decomposition. Overall, this additional condition ensures a well working
induction for deducing vertex decomposability. We postpone the precise
definition of star decomposability to Section~\ref{s:star_decomposability}.

Finally, we will utilize the fact that we are interested in star
decomposability of the complex $\sd^2 \K$ which is a barycentric subdivision of
another complex, namely $\sd^2 \K = \sd \L$ where $\L = \sd \K$. We will
introduce the notion of star decomposability in vertices which will mean that
we are removing only stars centered in vertices of $\sd \L$ which are
simultaneously vertices of $\L$ as in Figure~\ref{f:star_cover_overlap}. (Note
that vertices of $\L$ are faces of $\K$.) This brings one more advantage. We
will essentially need claims of the following spirit: 
If $\sd(\X)$ and $\sd(\Y)$ are star decomposable in vertices, then $\sd(\X *
\Y)$ is star decomposable in vertices as well (here $\X * \Y$ denotes the join
of $\X$ and $\Y$). In addition, we will also need to describe the order of the
star decomposition in vertices of $\sd(\X * \Y)$. Though it is probable that
analogous claims are valid also for star decomposability, vertex
decomposability and/or shellability, the notion of star decomposability in
vertices removes at least one layer of complications in the proof: It is just
sufficient to describe the order of the decomposition of $\sd(\X * \Y)$ as some
total order on $V(\X * \Y) = V(\X) \sqcup V(\Y)$ via a suitable way of
interlacing the total orders on $V(\X)$ and $V(\Y)$ (here $V(\X) \sqcup V(\Y)$
denotes the disjoint union of $V(\X)$ and $V(\Y)$).

\section{Preliminaries}
\label{s:prelim}
%

In this section, we briefly overview the standard terminology regarding
simplicial complexes, including some of the notions mentioned in the
introduction without the definition. We generally assume that the reader is familiar with
simplicial complexes. Thus the main purpose is to set up the notation. 


We work with finite \emph{abstract simplicial complexes}, that is, finite set systems $\K$
such that if $\sigma \in \K$ and $\sigma^\prime \subseteq \sigma$, then
$\sigma^\prime \in \K$. Elements of $\K$ are \emph{faces}; a $k$-face is a face
of \emph{dimension} $k$, that is, a face of size $k+1$. Vertices correspond to
$0$-faces of $\K$ (specifically, a vertex $v$ corresponds to a $0$-face
$\{v\}$); and the set of vertices is denoted $V(\K)$.
The dimension of $\K$
is the dimension of the largest face (or $-\infty$, if $\K$ is empty). 
The complex $\K$ is \emph{pure} if all
inclusion-wise maximal faces have the same dimension. 





A \emph{join} of two simplicial complexes $\K_1$ and $\K_2$ is the 
complex $\K_1 * \K_2 := \{\sigma_1 \sqcup \sigma_2\colon \sigma_1 \in \K_1, \sigma_2
\in \K_2\}$ where $\sqcup$ stands for disjoint union.\footnote{We can perform
the disjoint union of two sets $A$ and $B$ even if $A$ and $B$ are not disjoint.
The standard model in such case is to take $A \times \{1\} \cup B \times
\{2\}$.} In our inductive arguments, we will carefully distinguish the empty
complex $\emptyset$ and the complex $\{\emptyset\}$ 
containing a single face, which is $\emptyset$. Note that $\K * \emptyset =
\emptyset$, whereas $\K * \{\emptyset\} = \K$. 

Given a face $\sigma$ of $\K$, the \emph{link} of $\sigma$ in $\K$ is defined
as
$\lk(\sigma, \K) := \{\sigma^\prime \setminus \sigma\colon \sigma^\prime \in \K,
\sigma \subseteq \sigma^\prime\}$. 
The (closed) \emph{star} of $\sigma$ in $\K$
is defined as 
$\st(\sigma, \K) := \{\sigma' \in \K\colon \sigma' \cup \sigma \in \K\}$. 


The \emph{barycentric subdivision} of a simplicial complex $\K$ is the  simplicial complex
\begin{align*}
\sd \K := \{\{\sigma_1, \ldots, \sigma_n\}\colon \sigma_1, \ldots, \sigma_n \in \K,\
\emptyset \neq \sigma_1 \subsetneq \sigma_2 \subsetneq \cdots \subsetneq \sigma_n\}.
\end{align*}

The geometric idea behind the definition of barycentric subdivision is the
following: According to the definition, the vertices of $\sd \K$ are nonempty
faces of $\K$. Place a vertex of $\sd \K$ into the barycenter of the face it
represents in $\K$ (in the geometric realization of $\K$, which we did not define
here). Then $\sd \K$ represents a (geometric) subdivision of $\K$; see
Figure~\ref{f:single_barycentric}. (In the subsequent text, we will not need any details about
geometric realization of the barycentric subdivision. However, we will use this
geometric interpretation in motivating pictures.)


\begin{figure}
\begin{center}
\includegraphics{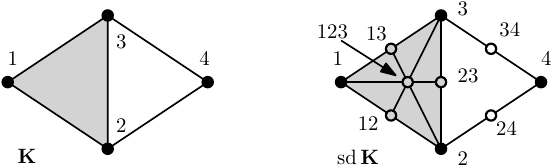}
  \caption{The barycentric subdivision $\sd \K$ of a complex $\K$. The notation
  on the right picture is simplified so that $12$ stands for $\{1,2\}$, etc.}
\label{f:single_barycentric}
\end{center}
\end{figure}

Note also that if $v$ is a vertex of $\K$, then $\{v\}$ is a vertex of $\sd
\K$. If there is no risk of confusion, we write $v$ instead of $\{v\}$ in
formulas such as $\lk(v , \sd \K)$.  We apply similar conventions to the second
barycentric subdivision, so we write $\lk(v, \sd^2 \K)$ instead of the cumbersome
$\lk(\{\{v\}\}, \sd^2 \K)$, or $\lk(\sigma, \sd^2 \K)$ instead of
$\lk(\{\sigma\}, \sd^2 \K)$ if $\sigma$ is a face of $\K$. 



\section{Star decomposability} \label{s:star_decomposability}

Given a simplicial complex $\X$ and a set $W \subseteq V(\X)$, we say that $W$
\emph{induces a star partition} of $\X$ if 
\begin{enumerate}[(i)]
\item $\X = \bigcup_{w \in W} \st(w, \X)$, and
\item any two distinct vertices $w_1, w_2 \in W$ are not neighbors in $\X$.
\end{enumerate}
An example of a set inducing a star partition is the set $\{w_1, w_2, w_3, w_4\}$
in Figure~\ref{f:sd}.

Now, let us assume that $W$ induces a star partition. Given a total order $\prec$ on $W$, $W' \subseteq W$, and $w \in W$, we set
$W'_{\succ w} := \{w' \in W'\colon w' \succ w\}$ and 
$W'_{\succeq w} := \{w' \in W'\colon w' \succeq w\}$. We will also use the notation
$$\st(W', \X) := \bigcup_{w' \in W'} \st(w', \X)$$
for an arbitrary subset $W'$ of
$V(\X)$. Furthermore given $x \in W$ and a set $W' \subseteq W$, we
define\footnote{The symbol $\O$ in the notation stands for the `overlap' of
$\lk(x, \X)$ and $\st(W', \X)$.}
\begin{equation}
\label{e:overlap}
\O(x, W') := \lk(x, \X) \cap \st(W', \X) = \lk(x, \X) \cap \bigcup_{w' \in W'}
\st(w', \X) = \lk(x, \X) \cap \bigcup_{w' \in W'} \lk(w', \X).
\end{equation}
See Figure~\ref{f:sd}. Note that this is the overlap mentioned in the
introduction. Occasionally, if we need to emphasize dependency on $\X$, we write $\O_{\X}(x,
W')$.


Now, we are ready to introduce star decomposability. Following the sketch in
the introduction, we want to introduce star decomposability of a simplicial
complex $\X$. However, in order to formulate all conditions correctly, we need
to state this definition for pairs.


\begin{definition}[Star decomposability]
\label{d:sd}
Let $(\X, X)$ be a pair where $\X$ is a simplicial complex which is pure and
  $k$-dimensional, $k \geq -1$ (that is, $\X \neq \emptyset$), and $X \subseteq
  V(\X)$. 
We inductively define \emph{star decomposability} of the pair $(\X, X)$. We also say that
$\X$ is \emph{star decomposable} if there is $X \subseteq V(\X)$ for which the pair
$(\X, X)$ is star decomposable.

For $k = -1$, the pair $(\{\emptyset\},\emptyset)$ is star decomposable.


If $k \geq 0$, then $(\X, X)$ is star decomposable, if there is a set $W \neq
\emptyset$
inducing a star partition and a total order $\prec$ on $W$ with the following
properties. 

\begin{description}
  \item[Order condition:] $X = W_{\succeq w'}$ for some
    $w^\prime \in W$.
  \item[Link condition:] For any vertex $w \in W$ except for the last vertex in
    the order $\prec$, there is a nonempty set $U = U(w) \subseteq V(\lk(w, \X))$ such that
   $\st(U,\lk(w, \X)) =   \O(w, W_{\succ w})$ and 
   the pair $(\lk(w, \X),U)$ is star decomposable.
  \item[Last vertex condition:] For the last vertex $\hat{w} \in W$ in the order
   $\prec$, the link $\lk(\hat{w}, \X)$ is star decomposable. 
\end{description}
If the order $\prec$ on $W$ satisfies the three conditions above, we say that
$\prec$ induces a \emph{star decomposition} of $(\X, X)$.

See Figure~\ref{f:sd} for an example.
\end{definition}

\begin{figure}
  \centering
  \includegraphics[page=1]{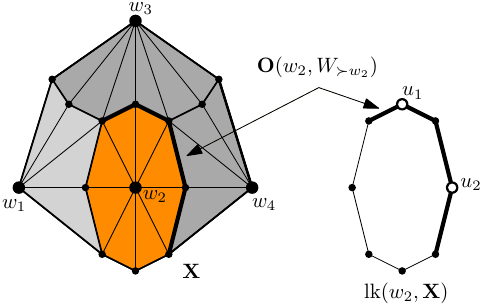}
  \caption{An example of the~star decomposition induced by the~set $W = \{w_1, w_2, w_3, w_4\}$
    with the~order $w_1 \prec w_2 \prec w_3 \prec w_4$ (left) and an example of the set
    $U(w_2) = \{u_1, u_2\}$ such that $\st(U(w_2),\lk(w, \X)) =
    \O(w_2, W_{\succ w_2})$ and the pair $(\lk(w_2, \X),U(w_2))$
   is star decomposable (right).}\label{f:sd}
\end{figure}

\newpage
\begin{remarks}
\label{r:sd}
\noindent
\begin{enumerate}[(i)]
\item Observe that the order condition implies $X \neq \emptyset$ if $k \geq 0$.
\item
In the definition above, we remark that if $\X$ is $k$-dimensional and pure,
    for $k \geq 0$, then for any $w \in V(\X)$, the link $\lk(w,\X)$ is
    $(k-1)$-dimensional and pure. Therefore, in the
    last two conditions, we indeed refer to star decomposability of a pure
    complex of
    smaller dimension.

    In addition, for any $W' \subseteq V(\X)$, $W' \neq \emptyset$, $\st(W', \X)$ is $k$-dimensional
    and pure. In particular, when replacing $\X$ with $\lk(w, \X)$, we get that
    $\O(w, W_{\succ w}) = \st(U, \lk(w, \X))$ is $(k-1)$-dimensional and pure.


%


\item If $k = 0$, then every pair $(\X, X)$ is star decomposable if and only
  if  $X \neq \emptyset$. Indeed, the only if part follows from (i). For the `if' part, we observe that we can set $W = V(\X)$ and we can use
    any order $\prec$ on $W$ such that $X = W_{\succeq w'}$ for some $w'$. 
  Both the link condition and the last vertex condition refer to star
  decomposability of $(\{\emptyset\}, \emptyset)$, which we assume.
    

  \item If $k=1$, then it is not difficult to show that $\X$ is star
    decomposable if and only if $\X$ is a connected bipartite graph. Note that
    requiring that $\X$ is connected is a must as we want to get that star
    decomposability implies vertex decomposability. Here is the place where the
    possibly slightly mysterious property `$X \neq \emptyset$ if $k \geq 0$' comes into
    the play. Indeed, this property and the link condition achieve that the overlap $\O(w, W_{\prec
    w})$ is nonempty, thus $\X$ must be a connected graph.
%
     

%
%
\end{enumerate}
\end{remarks}

\section{Star decomposability implies vertex decomposability.}


In this section, we want to describe how star decomposability implies vertex
decomposability. 
%
We start with a simple (folklore) lemma verifying that some order is a
shedding order (with respect to our convention that we extend the shedding
order also to the vertices of the last simplex). Given a simplicial complex $\X$, a
total (or partial) order $\prec$ on $V(\X)$, and $v \in V(\X)$, by $\X_{\succ v}$ we denote
the subcomplex of $\X$ induced by vertices that are greater than $v$. Similarly
$\X_{\succeq v}$ is induced by $v$ and the vertices that are greater than $v$.

\begin{lemma}
\label{l:is_shedding}
Let $\X$ be a pure $k$-dimensional simplicial complex, $k \geq 0$. Let $\prec$
be a total order on $V(\X)$. Then $\prec$ is a shedding order if and only if for every
vertex $v$ except for the last $k+1$ vertices, the link $\lk(v, \X_{\succeq v})$ is
vertex decomposable and $(k-1)$-dimensional, and $\X_{\succ v}$ is pure
$k$-dimensional. 
\end{lemma}

\begin{proof}
The `only if' part of the statement follows immediately from the definition of
vertex decomposability and the shedding order, thus we focus on the `if' part.

If $\X$ has $k+1$ vertices, then $\X$ is a $k$-simplex and we are done.
Otherwise, we proceed by induction on the number of vertices of $\K$.

Let $v_1$ be the first vertex in the order $\prec$. Then we need to check that
$\lk(v_1, \X_{\succeq v_1})$ is vertex decomposable and $(k-1)$-dimensional, which is
part of the assumptions. We also need to check that $\X - v_1 = \X_{\succ v_1}$ is vertex
decomposable and $k$-dimensional. Again, $k$-dimensional is part of the
assumptions, thus, it remains to check that $\X - v_1$ is vertex decomposable.
However, this follows from the induction applied to $\X_{\succ v_1}$ and $\prec$
restricted to $V(\X) \setminus \{v_1\}$. 
\end{proof}

Now, let $\X$ be a star decomposable simplicial complex, let $W$ be a subset of $V(\X)$ which induces a star
partition of $\X$ and let $\prec$ be a total order which induces a star decomposition
of $\X$. We will define a suitable partial order $\prec'$ on
$V(\X)$ extending $\prec$ such that the desired shedding order in the vertex
decomposition of $\X$ will follow $\prec'$. 

For arbitrary $v \in V(\X)$, let $p(v)$ be the last vertex in the
$\prec$ order among the vertices $w \in W$ such that $v \in \st(w, \X)$. 
In particular $p(w) = w$ for any $w \in W$.  
If we want to emphasize $\prec$, we write $p(v, \prec)$ (which will be used in
a single but important case of the proof of Theorem~\ref{t:sd_vd}). 
Now, we define $\prec'$ in the following way.  We set $v \prec' v'$ if $p(v)
\prec p(v')$ for $v, v' \in V(\X)$. In addition, we set $v \prec' w$ if $p(v) =
w$ and $v \neq w$. Finally, if $p(v) = p(v')$ and $v, v' \not\in W$, then $v$ and $v'$ are
incomparable in $\prec'$. We say that $\prec'$ is \emph{derived from} $\prec$.
An example of this order is given in Figure~\ref{f:Ps} where $P(w) = \{v \in
V(\X)\colon v \neq w, \ p(v) = w\}$ for $w \in W$; the elements in $P(w)$ are
incomparable.

\begin{figure}
\begin{center}
  \includegraphics[page=2]{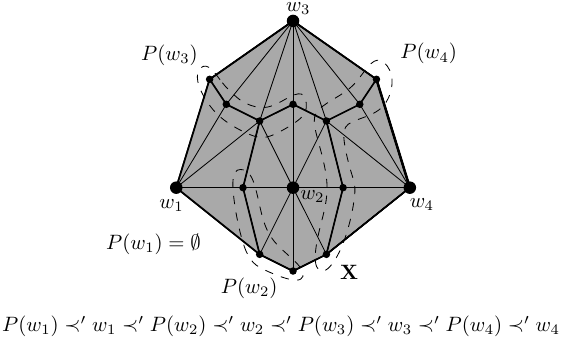}
  \caption{The set $P(w)$ and the auxiliary order $\prec'$ for the star
  decomposition in Figure~\ref{f:sd}.}  \label{f:Ps}
\end{center}
\end{figure}
We will often need that $\st(W_{\succ w}, \X)$ is an induced subcomplex of
$\X$ for $w \in W \setminus\{\hat w\}$:

\begin{lemma}
  \label{l:star_induced}
  Let $\X$ be a star decomposable complex, let $W$ be a subset of $V(\X)$ which
  induces a star
  partition of $\X$ and let $\prec$ be a total order on $W$ which induces a star
  decomposition
  of $\X$. Let $\prec'$ be the partial order on $V(\X)$ derived from $\prec$ and let $w \in W$ be different from
  the last vertex $\hat w$. Then $\st(W_{\succ w}, \X)$ is the induced
  subcomplex $\X_{\succ' w}$ of $\X$.
\end{lemma}

\begin{proof}
If $\dim \X = -1$, then the statement is void. If $\dim \X = 0$, the assertion
  easily follows from Remark~\ref{r:sd}(iii). Thus, we may assume $\dim \X \geq
  1$, which we will implicitly when referring to the link condition.

  Recall that $\st(W_{\succ w}, \X) =  \bigcup_{w^+ \in W_{\succ w}} \st(w^+,
  \X)$. It is easy to check the inclusion $\X_{\succ' w} \supseteq \bigcup_{w^+
  \in W_{\succ w}} \st(w^+,  \X)$ because $\st(w^+, \X) \subseteq \X_{\succ' w}$ for every $w^+ \succ w$. (Note that if $v$ is a
  neighbor of $w^+$ in $\X$, then $p(v) \succeq w^+ \succ w$. Thus, $v$ belongs
  to $V(\X_{\succ' w})$.) Therefore, it remains to show $\X_{\succ' w}
  \subseteq  \bigcup_{w^+ \in W_{\succ w}} \st(w^+, \X_{\succ' w})$.


  Let $\sigma \in \X_{\succ' w}$. For contradiction, let us assume that $\sigma
  \not\in \st(w^+, \X)$ for all $w^+ \succ w$. (In particular, $\sigma \neq
  \emptyset$.) Let $w^- \preceq w$ be the
  largest vertex in $W$ (according to the total order $\prec$) such that $\sigma \in \st(w^-, \X)$. 
  Such $w^-$ must exist because $W$ induces a star partition of $\X$. In
  addition, because $\sigma \in \X_{\succ' w}$ and $w^- \preceq w$, we get that
  $w^- \not \in \sigma$. Thus, $\sigma \in \lk(w^-, \X)$.

  Now, we use that $\X$ is star decomposable. Namely, we use the link
  condition for $w^-$. There is $U \subseteq V(\lk(w,\X))$ such that $\st(U,
  \lk(w^-,\X)) = \O(w^-, W_{\succ w^-})$ and the pair $(\lk(w^-, \X), U)$ is
  star decomposable. By the order condition for this pair, there is a set $Z
  \neq \emptyset$ inducing a star partition of $\lk(w^-, \X)$ and a total order
  $\lhd$ on $Z$ such that $U = Z_{\unrhd z'}$ for some $z' \in Z$.
  Because $\sigma \in \lk(w^-, \X)$ 
  and $Z$ induces a star partition of $\lk(w^-, \X)$ some
  vertex $v$ of $\sigma$ has to belong to $Z$. If $v \in U$, then $\sigma \in
  \st(U, \lk(w^-,\X)) = \O(w^-, W_{\succ w^-})$ which contradicts the fact that
  $w^-$ is the largest vertex such that $\sigma \in \st(w^-, \X)$. If $v \in Z
  \setminus U$, then $p(v) = w^-$ which contradicts $\sigma \in \X_{\succ' w}$.
\end{proof}

Now, we are ready to state and prove that star decomposability implies vertex
decomposability. As the reader may expect, the order $\prec'$ appears in the statement 
to allow a well working induction.

\begin{theorem}
\label{t:sd_vd}
 Let $\X$ be a star decomposable simplicial complex; 
 let $W$ be a subset of $V(\X)$ which induces a star
partition of $\X$; and let $\prec$ be a total order which induces a star decomposition
  of $\X$. Let $\prec'$ be the partial order on $V(\X)$ derived from $\prec$. 
  Then $\X$ is vertex decomposable in a shedding order extending
  $\prec'$.
\end{theorem}

\begin{proof}
  We prove the statement by induction on $k$, the dimension of $\X$. If $k = -1$, the complex
  $\{\emptyset\}$ is vertex decomposable according to the definition of vertex
  decomposability (it is regarded as a $-1$-simplex). Although it could be covered
  by the second induction step, we can observe that the case $k = 0$ is also easy
  as any order of removing vertices from a $0$-complex is a shedding order.

 Now, let us prove the theorem for some $k \geq 1$ assuming that it is valid
 for lower values.

 We first describe a total order $\prec''$ on $V(\X)$ extending $\prec'$. Then
 we verify that $\prec''$ is a shedding order. Recall that for $w \in W$, $P(w)$ is the
 set of vertices $v \in V(\X)$ such that $p(v) = w$ but $v \neq w$; see
  Figure~\ref{f:Ps}.
 To describe $\prec''$ 
 it remains to describe $\prec''$ on each $P(w)$ separately. We distinguish
 whether $w$ is the last vertex in $\prec$.


 If $w = \hat w$ is the last vertex, then $P(\hat w) = V(\lk(\hat w, \X))$. By the last vertex
 condition (for star decomposability) $\lk(\hat w, \X)$ is star decomposable,
 therefore vertex decomposable by induction as well. We set $\prec''$ on
 $P(\hat w)$ as an arbitrary shedding order of $\lk(\hat w, \X)$.

 If $w$ is not the last vertex, then $P(w) = V(\lk(w, \X)) \setminus V(\O(w,
 W_{\succ w}))$. By the link condition, the pair $(\lk(w, \X), U)$ is star
 decomposable where $U \subseteq V(\lk(w,\X))$ satisfies $\st(U, \lk (w, \X)) =
 \O(w,  W_{\succ w})$. 
  
  \begin{claim}
 \label{c:inductive_shedding}
   Let $w \in W$ be different from the last vertex $\hat w$. Then the link $\lk(w, \X)$ is vertex decomposable in
     some shedding order $\lhd''$ that starts on $P(w) = V(\lk(w, \X)) \setminus V(\O(w,
      W_{\succ w}))$ and then continues on $V(\O(w, W_{\succ w}))$.
  \end{claim}
 
  \begin{proof}
    Consider a set $Z \subseteq V(\lk(w,\X))$ inducing a star partition of
    $\lk(w,\X)$ and a total order $\lhd$ on $Z$ witnessing that the pair
    $(\lk(w, \X), U)$ is star decomposable. In particular, $U = Z_{\unrhd z'}$
    for some $z' \in Z$ by the order condition.
    Let $\lhd'$ be the partial
    order on $V(\lk(w, \X))$ derived from~$\lhd$. 
    By induction, $\lk(w, \X)$ is
      vertex decomposable in a shedding order $\lhd''$ extending $\lhd'$.

    In addition, by the link condition (on star decomposable $\X$) we get
    $$\O(w,  W_{\succ w}) = \st(U, \lk (w, \X)) = \st(Z_{\unrhd z'}, \lk (w,
    \X)).$$
    The vertices of $\st(Z_{\unrhd z'}, \lk (w, \X))$ are exactly the
    vertices of $\lk(w, \X)$ with $p(v, \lhd) \in Z_{\unrhd z'}$.
    Therefore all vertices in $V(\lk(w, \X)) \setminus V(\O(w,
	  W_{\succ w}))$ precede the vertices in $V(\O(w, W_{\succ w})) =
	  V(\st(Z_{\unrhd z'}, \lk (w, \X)))$ in the order $\lhd'$, a fortiori,
	  in the order $\lhd''$, as we need.
  \end{proof}

  
  Now, we set $\prec''$
 on $P(w)$ as the shedding order $\lhd''$ on $\lk(w, \X)$ from
 Claim~\ref{c:inductive_shedding}, restricted to
  $P(w)$; see Figure~\ref{f:order_Pw2}. 

\begin{figure}
\begin{center}
  \includegraphics[page=3]{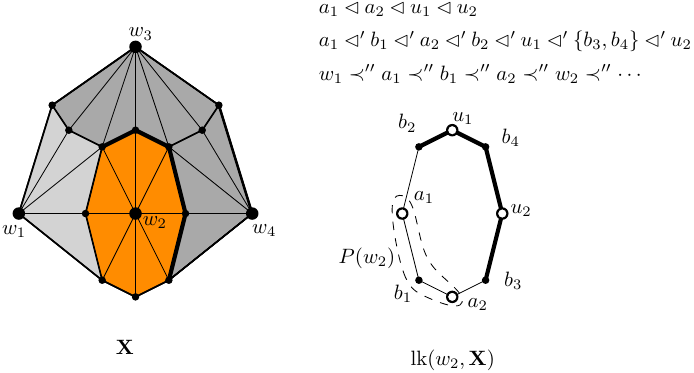}
  \caption{Setting up the order $\prec''$ on $P(w_2)$. The order $\lhd$ on
  $Z = \{a_1, a_2, u_1, u_2\}$ induces a star decomposition of $(\lk(w_2, \X), U)$
  where $U = \{u_1, u_2\}$. Then $\lhd'$ is the corresponding partial order on
  $V(\lk(w_2, \X))$ (similarly as $\prec'$ corresponds to $\prec$). Finally, we
  take a shedding order $\lhd''$ on $\lk(w_2, \X)$ extending $\lhd'$ (by induction)
  and restrict it to $P(w_2)$ obtaining $\prec''$.}
  \label{f:order_Pw2}
\end{center}
\end{figure}

 It remains to check that $\prec''$ is the required shedding order which we do
 via Lemma~\ref{l:is_shedding}. Namely,
 given a vertex $v \in V(\X)$ which is not one of the last $k+1$ vertices,
we need to check that 
 $\lk(v, \X_{\succeq'' v})$ is vertex decomposable and $(k-1)$-dimensional and
 that $\X_{\succ'' v}$ is pure $k$-dimensional. 
We distinguish whether $v \in W$.

\bigskip

\emph{Case 1}, $v \in W$:
 We observe that $v$ is not the last vertex $\hat w$ of $\prec$ as $\hat w$ is
 also the last vertex of $\prec''$. This allows to describe $\lk(v, \X_{\succeq'' v})$
 as an overlap.
\begin{claim}
\label{c:lk_o}
  $\lk(v, \X_{\succeq'' v}) = \O(v,
    W_{\succ v})$.
\end{claim}

\begin{proof}
According to the definition of the overlap, we have $\O(v,
    W_{\succ v}) = \lk(v, \X) \cap \st(W_{\succ v}, \X)$.

  First, let us assume that $\sigma \in \lk(v, \X) \cap \st(W_{\succ v}, \X)$.
  Each vertex $v'$ of $\st(W_{\succ v}, \X)$ satisfies $p(v') \succ v$ which
  implies $v' \succ'' v$. Therefore, each vertex of $\sigma \cup \{v\}$ belongs
  to $V(\X_{\succeq'' v})$. Because $\sigma$ simultaneously belongs to $\lk(v,
  \X)$, we get that it belongs to $\lk(v, \X_{\succeq'' v})$.

  Now, for the second inclusion, let us assume that $\sigma \in \lk(v,
  \X_{\succeq'' v})$.
Immediately, $\sigma \in \lk(v, \X)$.
  Because $\sigma \in \X_{\succ' v} = \X_{\succ'' v}$, Lemma~\ref{l:star_induced} gives
    $\sigma \in \st(W_{\succ v}, \X)$.
%
%
\end{proof}

By Claim~\ref{c:lk_o}, $\lk(v, \X_{\succeq'' v}) = \O(v,
  W_{\succ v})$ which is $(k-1)$-dimensional by Remark~\ref{r:sd}(ii).
  In addition, $\lk(v, \X_{\succeq'' v})$ is vertex decomposable, as we
 checked that $\lk(v, \X)$ is vertex decomposable in some shedding order
 starting with $P(v) = V(\lk(v, \X)) \setminus V(\O(v,  W_{\succ v}))$ and
 continuing with $V(\O(v,  W_{\succ v}))$; see Claim~\ref{c:inductive_shedding}. Also $\X_{\succ'' v} = \X_{\succ' v}
 = 
 \st(W_{\succ v}, \X)$ by Lemma~\ref{l:star_induced}.
 Therefore $\X_{\succ'' v}$ is pure $k$-dimensional by
  Remark~\ref{r:sd}(ii). This finishes Case~1.

\bigskip

\emph{Case 2}, $v \not\in W$:
Let $w := p(v) \in W$. Note that, in particular, $w \succ' v$.  
We first check that $\lk(v,
   \X_{\succeq'' v})$ is vertex decomposable and $(k-1)$-dimensional. This will follow from the following two claims.

\begin{claim}
\label{c:linkv_join}
  The link $\lk(v, \X_{\succeq'' v})$ is the join of $w$ and
  $\lk(\{v,w\},\X_{\succeq'' v})$.
\end{claim}


\begin{proof}
  The link $\lk(\{v,w\},\X_{\succeq'' v})$ consists of simplices $\sigma \in
  \X_{\succeq'' v}$ satisfying $v, w \not\in \sigma$, and $\sigma \cup \{v,w\}
  \in \X_{\succeq'' v}$. Therefore, the join of $w$ and
  $\lk(\{v,w\},\X_{\succeq'' v})$, considered as a subcomplex of $\X_{\succeq''
  v}$, consists of simplices $\sigma \in
  \X_{\succeq'' v}$ satisfying
  \begin{equation}
    \label{e:w*link}
  v \not\in \sigma, \hbox{ and } \sigma \cup \{v,w\}
  \in \X_{\succeq'' v}.
  \end{equation}
  On the other hand, $\lk(v, \X_{\succeq'' v})$ consists of simplices $\sigma \in
  \X_{\succeq'' v}$ satisfying
  \begin{equation}
    \label{e:linkv}
  v \not\in \sigma, \hbox{ and } \sigma \cup \{v\}
  \in \X_{\succeq'' v}.
  \end{equation}

A simplex $\sigma \in \X_{\succeq'' v}$ satisfying~\eqref{e:w*link} immediately
  satisfies~\eqref{e:linkv} as well. Thus, it remains to consider a simplex
  $\sigma \in \X_{\succeq'' v}$ satisfying~\eqref{e:linkv}; and to show that it
  satisfies~\eqref{e:w*link}.

 
  First, we want to deduce that $\sigma \cup \{v\}$ belongs to $\st(w', \X)$
  for some $w' \succeq w$. If $w$ is the first vertex of $W$ in the order
  $\prec$, then this claim follows from the fact that $W$ induces a star
  partition of $\X$. If $w$ is not the first vertex of $W$, let $w^-$ be the
  vertex that immediately precedes $w$ in the order $\prec$. Note that $\sigma
  \in \X_{\succ' w^-}$. By
  Lemma~\ref{l:star_induced}, $\sigma \cup
  \{v\}$ belongs to
  $\st(w', \X)$ for some $w' \succ w^-$, that is, $w' \succeq w$ as required.
  
  Now, because $p(v) = w$, the only option is that $w' = w$. Therefore,
  $\sigma \cup \{v\} \in \st(w, \X)$; that is, $\sigma \cup \{v,w\} \in \X$.
  Because all vertices of $\sigma \cup \{v,w\}$ belong to $\X_{\succeq'' v}$,
  $\sigma$ satisfies~\eqref{e:w*link}.
\end{proof}

\begin{claim}
\label{c:linkv_decomposable}
  The link $\lk(\{v,w\},\X_{\succeq'' v})$ 
  is vertex decomposable and $(k-2)$-dimensional.
\end{claim}

\begin{proof}
We will deduce the claim from the `only if' part of Lemma~\ref{l:is_shedding}
  used with the pure $(k-1)$-dimensional complex $\lk(w, \X)$ and the shedding
  order $\lhd''$, coming from Claim~\ref{c:inductive_shedding}. Let us recall
  that $\prec''$ is defined so that it coincides with $\lhd''$ on $\lk(w, \X)$
  restricted to $P(w)$. Because $v \in P(w)$, we in particular get that $\lk(w,
  \X)_{\succeq'' v} = \lk(w, \X)_{\unrhd'' v}$. 
  
  In order to apply Lemma~\ref{l:is_shedding}, we also check that
  $v$ is not among the last $k$ vertices of the aforementioned shedding
  $\lhd''$ of $\lk(w,\X)$. 
  If $w = \hat w$, we get this because we assume that $v$ is not among the last
 $k+1$ vertices in the $\prec''$ order on $V(\X)$ (the last one is $\hat w$,
 and the vertices of $P(\hat w)$ immediately precede). If $w \neq \hat w$ we
 get this because the overlap $\O(w, W_{\succ w})$ is $(k-1)$-dimensional (see
  Remark~\ref{r:sd}(ii)), and the vertices of this overlap belong to $V(\lk(w,\X))$ while they do not belong to $P(w)$.

  Now, using Lemma~\ref{l:is_shedding} as explained above, we get that $\lk(v,
  \lk(w,\X)_{\succeq'' v}) =
  \lk(v,
  \lk(w,\X)_{\unrhd'' v})$ is vertex decomposable and $(k-2)$-dimensional.
  Finally, $\lk(v,
  \lk(w,\X)_{\succeq'' v}) = \lk(\{v,w\},\X_{\succeq'' v})$ because
  $\X_{\succeq'' v}$ is an induced subcomplex of $\X$. 
\end{proof}


It follows immediately from Claims~\ref{c:linkv_join} and~\ref{c:linkv_decomposable} that 
$\lk(v, \X_{\succeq'' v})$ is $(k-1)$-dimensional. In addition, because the join of two vertex decomposable complexes is vertex
 decomposable~\cite[Proposition~2.4]{provan-billera80}, we also get that $\lk(v, \X_{\succeq'' v})$ is vertex decomposable.
 
Finally, we need to check that $\X_{\succ''v}$ is pure $k$-dimensional. We need
one more claim; see also Figure~\ref{f:pure_union}.

\begin{figure}
\begin{center}
  \includegraphics[page=4]{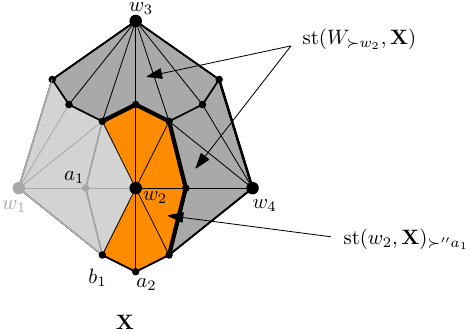}
  \caption{The complex $\X_{\succ'' a_1}$ as the union of $\st(W_{\succ w_2}, \X)$
  and $\st(w_2, \X)_{\succ'' a_1}$. Here we use the order $\succ''$ from
  Figure~\ref{f:order_Pw2}.}
  \label{f:pure_union}
\end{center}
\end{figure}

\begin{claim}
\label{c:pure_union}
 If $w \neq \hat w$, then $\X_{\succ''v} = \st(W_{\succ w}, \X) \cup \st(w,   \X)_{\succ'' v}$. If, $w = \hat w$, then $\X_{\succ''v} =
  \st(w,   \X)_{\succ'' v}$.
\end{claim}

\begin{proof}
  If $w \neq \hat w$, then $\st(W_{\succ w}, \X) = \X_{\succ' w} = \X_{\succ'' w}$ by
  Lemma~\ref{l:star_induced}.  Therefore it is sufficient to show that
  every $\sigma \in \X_{\succ''v}$ which contains a vertex $v'$ with $v'
  \preceq w$ belongs to $\st(w,   \X)_{\succ'' v}$. This will resolve both cases,
  $w = \hat w$ and $w \neq \hat w$, simultaneously. The ideas in the reminder
  of the proof are very similar to the ideas in the proof of
  Claim~\ref{c:linkv_join}.
  
 First, we check that $\sigma \in \st(w',   \X)$ for some $w' \succeq w$. If
  $w$ is the first vertex of $W$, then this follows from the fact that $W$
  induces a star decomposition of $\X$. If $w$ is not the first vertex of $W$,
  let $w^-$ be the vertex of $W$ that immediately precedes $w$. By
  Lemma~\ref{l:star_induced}, $\st(W_{\succ w^-}, \X) = \X_{\succ'' w^-}$.
  Because $\sigma \in \X_{\succ'' w^-}$, this implies that there is $w' \succ
  w^-$ with $\sigma \in \st(w',   \X)$.

  On the other hand,  $\sigma$ cannot belong to $\st(w'',   \X)$ with $w''
  \succ w$ as $\sigma$ contains $v'$ with $v' \preceq w$. Therefore, $w' = w$. 
  Given that $\st(w,  \X)_{\succ'' v} = \X_{\succ''  v} \cap \st(w,   \X)$, we
  deduce that $\sigma \in \st(w,  \X)_{\succ'' v}$.
\end{proof}

The union of two pure $k$-dimensional complexes is a pure $k$-dimensional
complex. Therefore, due to Claim~\ref{c:pure_union}, it remains to check
that $\st(W_{\succ w}, \X)$ and $\st(w,   \X)_{\succ'' v}$ are pure
$k$-dimensional (the former case applies only if $w \neq \hat w$).

Checking that $\st(W_{\succ w}, \X)$ is pure $k$-dimensional is easy; see
Remark~\ref{r:sd}(ii).

For checking that $\st(w, \X)_{\succ'' v}$ is pure $k$-dimensional, we need
that $\lk(w, \X)_{\succ'' v}$ is pure $(k-1)$-dimensional. Because $v \in
P(w)$, $\lk(w, \X)_{\succ'' v} = \lk(w, \X)_{\unrhd'' v}$ where $\lhd''$ is
the shedding of $\lk(w, \X)$ as introduced below
Claim~\ref{c:inductive_shedding}.
This means that
$\lk(w, \X)_{\succ'' v}$ is an intermediate step in the shedding $\unrhd''$ of
$\lk(w, \X)$. If we realize that $v$ is not among the last $k$
vertices of the order $\lhd''$ on $\lk(w, \X)$, then we can deduce that $\lk(w,
\X)_{\succ'' v}$ is pure and $(k-1)$-dimensional. 

If $w \neq \hat w$, then $\lk(w, \X)_{\succ'' v}$ still contains the overlap $\O(w,
W_{\succ} w)$ which is $(k-1)$-dimensional by Remark~\ref{r:sd}(ii). If $w =
\hat w$, then we assume that $v$ is not among the last $k+1$ vertices of
$\prec''$ while $\prec''$ and $\lhd''$ coincide on $P(\hat w)$ and the vertices
of $P(\hat w)$ immediately precede $\hat w$ in $\prec''$. This finishes Case~2
and thereby the proof of the theorem.
%
%
%
%
\end{proof}

\section{Star decomposability in vertices}

\paragraph{Star decomposability of a barycentric subdivision.} In our approach,
we will need to consider the star decomposability of the barycentric subdivision
$\sd(\X)$ of a complex $\X$. In fact, we will consider only a special type of
star decomposition of $\sd(\X)$ using only stars of vertices of $\X$, that
is, the faces of $\X$ which are actually vertices of $X$. For a well working
induction, we will need that this property is kept also in the link condition
and the last vertex condition of Definition~\ref{d:sd}. For stating this
precisely, first, we need a more explicit description of $\lk(\vartheta, \sd(\X))$ if
$\vartheta$ is a face (possibly a vertex) of $\X$.

\begin{lemma}
\label{l:lk_sd}
Let $\vartheta$ be a face of a simplicial complex $\X$, then
$$\lk(\vartheta, \sd \X) \cong \sd \partial \vartheta * \sd \lk(\vartheta, \X).$$
In particular, if $x$ is a vertex of $\X$, then 
$$\lk(x, \sd \X) \cong \sd \lk(x, \X).$$
\end{lemma}


\begin{proof}
  We will construct a simplicial isomorphism $$\Psi \colon V(\lk(\vartheta, \sd
  \X)) \to V(\sd  \partial \vartheta * \sd \lk(\vartheta, \X)).$$
  
First, we observe that 
$$
V(\sd  \partial \vartheta * \sd \lk(\vartheta, \X)) = V(\sd  \partial
\vartheta) \sqcup V(\sd \lk(\vartheta, \X)) = \partial \vartheta \sqcup
\lk(\vartheta, \X).
$$

 Next, we realize that the vertices of $\lk(\vartheta, \sd \X)$ are all the faces $\lambda \neq \emptyset, \vartheta$ of
  $\X$ such that $\{\lambda, \vartheta\}$ forms a simplex of $\sd \X$, that is,
  either $\emptyset \neq \lambda \subsetneq \vartheta$ or $\vartheta \subsetneq
  \lambda$. Thus, we can define $\Psi$ in the following way

$$
\Psi(\lambda) = 
\begin{cases}
  \lambda \in \partial \vartheta & \hbox{ if } \emptyset \neq \lambda \subsetneq
  \vartheta,\\
  \lambda \setminus \vartheta  \in \lk(\vartheta, \X) & \hbox{ if }
  \vartheta \subsetneq \lambda.\\
\end{cases}
$$
  
From the description above, it immediately follows that $\Psi$ is a bijection.
It is also routine to check that $\Psi$ is a simplicial isomorphism. Indeed, 
a simplex of $\lk(\vartheta, \sd \X)$ is a collection $\{\alpha_1, \dots,
\alpha_k, \beta_1, \dots, \beta_\ell\}$ satisfying
$$
\emptyset \neq \alpha_1 \subsetneq \cdots \subsetneq \alpha_k \subsetneq
\vartheta \subsetneq \beta_1 \subsetneq \cdots \subsetneq \beta_{\ell}.
$$
Such a simplex maps to a simplex $\{\alpha_1, \dots, \alpha_k, \beta_1
\setminus \vartheta , \dots, \beta_\ell \setminus \vartheta\}$ of $\sd  \partial
\vartheta * \sd \lk(\vartheta, \X)$ and the inverse map works analogously (note
that $\beta_i \setminus \vartheta$ is disjoint from $\vartheta$ whereas
$\alpha_i$ are subsets of $\vartheta$).
\end{proof}

Now, we extend the isomorphism above to certain pairs; for the statement, recall that $\O(x,W)$ is
defined via formula~\eqref{e:overlap}.

\begin{lemma}
\label{l:links_in_link}
Let $x$ be a vertex and $W$ a subset of vertices of the simplicial complex $\X$
such that $x \not\in W$. 
Then 
$$\left(\lk(x, \sd \X), \O_{\sd \X}(x, W)\right) \cong \left(\sd \lk(x,\X),
  \st(W \cap V(\lk(x,\X)),
\sd\lk(x,\X))\right).$$
\end{lemma}


Though the formula in Lemma~\ref{l:links_in_link} may seem complicated
at first sight, it has a nice geometric interpretation. All objects are
subcomplexes of $\sd \X$ and the isomorphism in the formula pushes the pair on
the left hand side farther away from $x$; see Figure~\ref{f:pair_s}.

\begin{figure}
\begin{center}
  \includegraphics{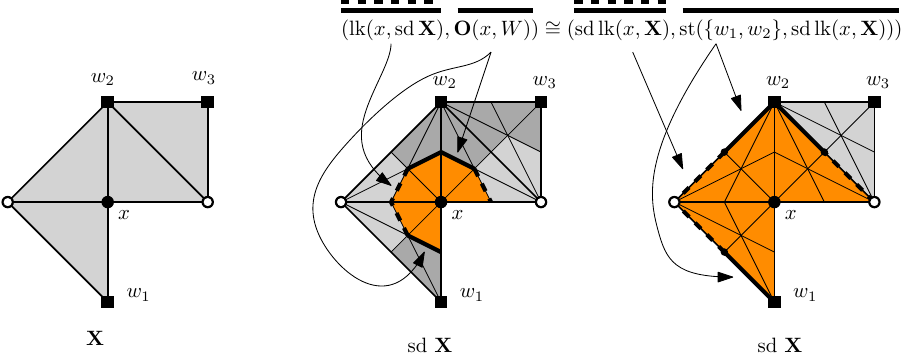}
  \caption{Isomorphism from Lemma~\ref{l:links_in_link} with $W =
  \{w_1,w_2,w_3\}$. The left hand side of the formula in
  Lemma~\ref{l:links_in_link} is depicted in the middle picture and the right
  hand side is in the right picture. Note that $W \cap V(\lk(x, \X)) =
  \{w_1,w_2\}$ as $w_3$ does not belong to $\lk(x, \X)$.}
\label{f:pair_s}
\end{center}
\end{figure}

\begin{proof}
  From Lemma~\ref{l:lk_sd} we have a simplicial isomorphism
  $\Psi$ from $\lk(x, \sd \X)$ to $\sd \lk(x,\X)$. Therefore, it remains to show that
  $\Psi$ maps $\O_{\sd \X}(x, w) := \lk(x, \sd \X) \cap \lk(w,\sd \X)$ to $\st(w,
  \sd \lk(x,\X))$ for $w \in W \cap V(\lk(x,\X))$, where we use the explicit
  $\Psi$ from the proof of Lemma~\ref{l:lk_sd}, and that $\O_{\sd \X}(x, w) =
  \emptyset$ for $w \in W \setminus V(\lk(x,\X))$.
  (Note that $\O_{\sd \X}(x,W) =  \bigcup_{w \in W} \O_{\sd \X}(x,w)$.)

  The faces of $\O_{\sd \X}(x,w)$ are collections $\{\beta_1, \dots, \beta_{\ell}\}$ of
  faces of $\X$ satisfying
$$
\{x, w\} \subseteq \beta_1 \subsetneq \cdots \subsetneq \beta_{\ell}.
$$
Let us emphasize that the first inclusion need not be strict. Therefore,
$\O_{\sd X}(x,w)$ is non-empty if and only if $\{x, w\} \in \X$, that is, if and only
if $w \in W \cap V(\lk(x,\X))$ as required. In sequel, we assume that $w \in W
\cap V(\lk(x,\X))$.

The collections $\{\beta_1, \dots, \beta_{\ell}\}$
are mapped under $\Psi$ to $\{\beta_1\setminus \{x\}, \dots, \beta_{\ell}
\setminus \{x\}\}$ satisfying the same condition due to the description of
$\Psi$ in the proof of Lemma~\ref{l:lk_sd}. Setting $\gamma_j =
\beta_j\setminus\{x\}$ we get
$$
\{w\} \subseteq \gamma_1 \subsetneq \cdots \subsetneq \gamma_{\ell}
$$
for $\gamma_j$ not containing $x$ but such that $\gamma_j \cup \{x\}$ is a face of $\X$, which is
exactly a description of $\st(w, \sd(\lk(x,\X)))$.
\end{proof}


Now, we can define star decomposibility in vertices:

\begin{definition}[Star decomposability in vertices]
\label{d:sdv}
Let $\X$ be a pure, $k$-dimensional simplicial complex, $k \geq -1$ and let $X
\subseteq V(\X)$. 
We inductively
define \emph{star decomposability in vertices} of the pair $(\sd \X, X)$. We also say that
$\sd \X$ is \emph{star decomposable in vertices} if the pair $(\sd \X, V(\X))$
is star decomposable in vertices.

If $k = -1$, then $(\sd \{\emptyset\}, \emptyset) =
(\{\emptyset\}, \emptyset)$ is star decomposable in vertices. (This is the same as star decomposability in this case.)

If $k \geq 0$, then $(\sd \X, X)$ is star decomposable in vertices, if there
is a total order $\prec$ on the set $V(\X)$, inducing a star partition of
$\sd \X$, with the following properties.\footnote{Note that $V(\X)$ induces a
star partition of $\sd \X$ for an arbitrary complex $\X$.}

\begin{description}
  \item[Order condition:] $X = V(\X)_{\succeq w'}$ for some $w' \in
    V(\X)$.  
  \item[Link condition:] For any vertex $w \in V(\X)$ except for the last vertex in
    the order $\prec$, the pair $\left(\sd \lk(w,\X), V(\lk(w, \X))_{\succ w}
     \right)$ is star decomposable in vertices.
   \item[Last vertex condition:] For the last vertex $\hat x \in V(\X)$ in the order
   $\prec$, the link $\sd \lk(\hat x,\X)$ is star decomposable in vertices. 
\end{description}
If the order $\prec$ on $W$ satisfies the three conditions above, we say that
$\prec$ induces a \emph{star decomposition of $(\sd \X, X)$ in vertices}.
\end{definition}


Lemma~\ref{l:links_in_link} implies the following proposition.

\begin{proposition}
Let us assume that the pair $(\sd \X, X)$ is star decomposable in vertices,
then it is star decomposable. 
\end{proposition}


\begin{proof}
We check that the order condition, the link condition and the last vertex
condition in Definition~\ref{d:sd} imply the corresponding
conditions in Definition~\ref{d:sdv}. The rest of the proof is a straightforward
induction given that in dimensions $-1$ and $0$ the notions coincide.

The order condition in Definitions~\ref{d:sd} and~\ref{d:sdv} is actually
identical.

For checking the link condition in Definition~\ref{d:sd}, for a given $w \in
V(\X)$ we need to find a set $U \subseteq V(\lk(w, \sd \X))$ such that
(i) $\st(U,\lk(w, \sd \X)) =   \O_{\sd \X}(w, V(\X)_{\succ w})$ and (ii)
   the pair $(\lk(w, \sd \X),U)$ is star decomposable in vertices (therefore star
   decomposable by induction). By Lemma~\ref{l:links_in_link} we have an
   isomorphism $\Psi$ mapping the pair 
   $$\left(\lk(w, \sd \X), \O_{\sd \X}(w, V(\X)_{\succ w})\right)$$ to the pair
   $$
   \left(\sd \lk(w,\X),  \st(V(\lk(w,\X))_{\succ w},
        \sd\lk(w,\X))\right),$$
 using that $V(\X)_{\succ w} \cap V(\lk(w,\X)) = V(\lk(w,\X))_{\succ w}$.
 We set $U := \Psi^{-1}(V(\lk(w,\X))_{\succ w})$, then (i) follows
 immediately from the isomorphism above. On the other hand, $(\lk(w, \sd
 \X),U)$ is isomorphic to $\left(\sd \lk(w,\X), V(\lk(w, \X))_{\succ w}
 \right)$ by applying $\Psi$. Therefore, (ii) indeed follows from the link
 condition of Definition~\ref{d:sdv}.

 Finally the last vertex condition of Definition~\ref{d:sdv} implies the
 same condition of Definition~\ref{d:sd} via Lemma~\ref{l:lk_sd} (and
 the induction).
\end{proof}

\paragraph{Merging orders inducing a star decomposition in vertices.}

Given simplicial complexes $\X$ and $\Y$ such that $\sd(\X)$ and
$\sd(\Y)$ are star decomposable in vertices, we want to provide an order on
$V(\X) \sqcup V(\Y)$ which induces a star decomposition in vertices of $\sd(\X
* \Y)$. For the proof of our main result we need some flexibility how to merge
 the orders on $V(\X)$ and $V(\Y)$. First we provide a recipe that works in
 general but does not give all we need. This is the contents of forthcoming
 Proposition~\ref{p:join}. Then we also provide a more specific recipe which 
 gives more under additional assumptions on $\Y$ (see
 Proposition~\ref{p:join_nice}).

\begin{proposition}
\label{p:join}
  Let $\X$ and $\Y$ be pure simplicial complexes such that $\sd(\X)$ and $\sd(\Y)$ are
  star decomposable in vertices. 
Let $\prec$ be an arbitrary total order on $V(\X) \sqcup V(\Y)$ satisfying that
\begin{enumerate}[(i)]
  \item the restriction of $\prec$ to $V(\X)$ induces a star decomposition in
    vertices of
    $\sd(\X)$,
  \item the restriction of $\prec$ to $V(\Y)$ induces a star decomposition in
    vertices of
    $\sd(\Y)$,
  \item if both $\X$ and $\Y$ are nonempty, then the last two elements in $\prec$ are the last element of $V(\X)$ and
    the last element of $V(\Y)$ (in arbitrary order).
\end{enumerate}
  Then $\sd(\X * \Y)$ is star decomposable in vertices in the order $\prec$ on $V(\X
  * \Y) = V(\X) \sqcup  V(\Y)$. 
\end{proposition}

\begin{corollary}
\label{c:join_pairs}
 Let $\X$ and $\Y$ be simplicial complexes and $X \subseteq V(\X)$, $Y
  \subseteq V(\Y)$. 
  Assume that the
 pairs $(\sd \X, X)$ and $(\sd \Y, Y)$ are star
 decomposable in vertices. Then the pair $(\sd (\X * \Y), X \sqcup Y)$
 is star decomposable in vertices as well. In addition, if $|Y| = 1$,
 then the pair $(\sd (\X * \Y), Y)$ is star decomposable in vertices.
\end{corollary}

%
%
%

\begin{proof}[Proof of Corollary~\ref{c:join_pairs}]

    First, let us assume that $X = \emptyset$. Because $(\sd \X, X)$ is
      star decomposable, we deduce that $\X = \{\emptyset\}$. Consequently,
      $(\sd (\X * \Y), X \sqcup Y) = (\sd \Y, Y)$, which is star decomposable
    in vertices. Similarly, we resolve the case $Y = \emptyset$.


    Now we can assume $X,Y \neq \emptyset$.  Let $\prec_{\X}$ be a total order on $V(\X)$ inducing a star decomposition of $(\sd \X,
  X)$ in vertices and let $\prec_{\Y}$ be a total order on $V(\Y)$ inducing a star
decomposition of $(\sd \Y, Y)$ in vertices. Let $\hat x$ be the
last vertex of $V(\X)$ in $\prec_{\X}$ and $\hat y$ be the 
last vertex of $V(\Y)$ in $\prec_{\Y}$. Necessarily, $\hat x \in X$ and $\hat y
\in Y$ as $X,Y \neq \emptyset$.

We define a total order $\prec$ on $V(\X) \sqcup V(\Y)$ so that we consider the
vertices of $V(\X) \sqcup V(\Y)$ in the order $[V(\X) \setminus X, V(\Y)
\setminus Y, X \setminus \{\hat x\}, Y \setminus \{\hat y\}, \hat x, \hat y]$,
where the individual sets $V(\X) \setminus X$, $V(\Y) \setminus Y$, $X \setminus
\{\hat x\}$, and $Y \setminus \{\hat y\}$ are sorted according to $\prec_{\X}$
and $\prec_{\Y}$ respectively. Then $\prec$ satisfies the assumptions of
Proposition~\ref{p:join}. Therefore, $\sd(\X * \Y)$ is star decomposable in
vertices in the order $\prec$. 

Given that $\st(X \sqcup Y,\sd (\X * \Y)) = \st((V(\X) \sqcup V(\Y))_{\succeq
z}, \sd (\X   * \Y))$ where $z$ is the first vertex of $X \cup Y$ in $\prec$,
we deduce that $\prec$ gives also a star decomposition of $(\sd (\X * \Y),
X \sqcup Y)$ in vertices.

Finally, if $|Y| = 1$, then $Y = \{\hat y\}$. Thus $\st(Y,\sd (\X
* \Y)) = \st((V(\X) \sqcup V(\Y))_{\succeq \hat y}, \sd (\X   * \Y))$ which
means that $\prec$ gives a star decomposition of $(\sd (\X * \Y),
Y)$ in vertices as well. 
\end{proof}



\begin{proof}[Proof of Proposition~\ref{p:join}.]
%
%
  First, similarly as in the previous proof, the statement is trivial if $\X = \{\emptyset\}$ or $\Y =
\{\emptyset\}$ as a join with $\{\emptyset\}$ yields the same complex.
Therefore, we can assume $\X, \Y \neq \{\emptyset\}$. In particular, the
item~(iii) of the statement is non-void. 

Now, we prove the proposition by induction on $\dim (\X * \Y)$. The start of
the induction, when $\dim (\X * \Y) \leq 0$, is covered by the observation above.

We are given the order $\prec$ on $V(\X * \Y)$; therefore it remains to check
the order condition, the link condition and the last vertex condition.  

\medskip

As we check star decomposability of $\sd (\X * \Y)$, that is, the pair $(\sd (\X
* \Y),V(\X) \sqcup V(\Y))$, the order condition is trivial. (It is sufficient to take
the first vertex of $V(\X) \sqcup V(\Y)$ for checking the order condition.)

\medskip

For checking the link condition, we consider arbitrary $x \in V(\X) \sqcup
V(\Y)$ distinct from the last vertex. Without loss of generality, we can assume $x \in V(\X)$ as the argument
is symmetric for a vertex from $V(\Y)$.
We need to check star decomposability of the pair
$$
(\sd(\lk(x, \X * \Y)), V(\lk(x, \X * \Y))_{\succ x}).
$$
Given that $x \in V(\X)$, this equals
\begin{equation}
\label{eq:pair_link_join}
(\sd(\lk(x, \X) * \Y), (V(\lk(x, \X)) \sqcup V(\Y))_{\succ x}).
\end{equation}


From the assumption on star decomposability of $\sd \Y$ in the order $\prec$, we deduce
that the pair
\begin{equation}
\label{eq:sd_Y}
(\sd(\Y),V(\Y)_{\succ x})
\end{equation}
is star decomposable in vertices as long as $V(\Y)_{\succ x}$ is nonempty.
However, $V(\Y)_{\succ x}$ is indeed nonempty as $x$ is not the last vertex of
$V(\X) \sqcup V(\Y)$ in $\prec$ whereas there is a vertex from $V(\Y)$ among
the last two vertices.


From the assumption on star decomposability of $\X$ in the order $\prec$,
checking the link condition gives that the pair
\begin{equation}
\label{eq:sd_X}
(\sd \lk(x, \X), V(\lk(x, \X))_{\succ x}
)
\end{equation}
is star decomposable in vertices if $x$ is not the last vertex of $V(\X)$. 
Therefore, if $x$ is not the last vertex of $V(\X)$, we will use the induction. 
From Corollary~\ref{c:join_pairs} for pairs~\eqref{eq:sd_X} and~\eqref{eq:sd_Y}
we deduce that the pair 
in~\eqref{eq:pair_link_join} is indeed star decomposable in vertices as
required. (Note that this is a correct use of the induction as we deduced
Corollary~\ref{c:join_pairs} from Proposition~\ref{p:join} in the same dimension.)

It remains to consider the case when $x$ is a last vertex of $V(\X)$. In this
case, $x$ is the second to last vertex of $V(\X) \sqcup V(\Y)$. Let $\hat y$ be
the last vertex of $V(\Y)$, that is, the last vertex of $V(\X) \sqcup V(\Y)$ as
well. Then the pair~\eqref{eq:pair_link_join} simplifies to
$$
(\sd(\lk(x, \X) * \Y), \{\hat y\}).
$$
Now, we can use Corollary~\ref{c:join_pairs} again with pairs $(\sd \lk(x, \X),
 V(\lk(x, \X)))$ and $(\sd(\Y), \{\hat y\})$, using the `in
 addition' part.

\medskip

Finally, it remains to check the last vertex condition. Let us therefore assume
that $\hat x$ is the last vertex of $V(\X) \sqcup V(\Y)$. Again, we can without loss
of generality assume that $\hat x \in V(\X)$. We need to check star
decomposability in vertices of $\sd \lk(\hat x, \X * \Y) = \sd (\lk(\hat x, \X) *
\Y)$. By the last vertex condition on $\sd(\X)$ we get that $\sd \lk(\hat x,
\X)$ is star decomposable in vertices. Therefore, by the induction applied to $\sd
(\lk(\hat x, \X))$ and $\sd \Y$, we get that $\sd (\lk(\hat x, \X) * \Y)$ is star
decomposable in vertices as required.
\end{proof}

Now, we state a more specialized version of Proposition~\ref{p:join} with an
additional condition on homology. Let us recall that given a simplicial
complex $\Y$ and $Y \subseteq V(\sd \Y)$, the star $\st(Y, \sd \Y)$ is defined
as $\bigcup_{v \in Y} \st(v, \sd \Y)$. Following our convention of neglecting a
difference between $v \in V(\Y)$ and $\{v\} \in V(\sd \Y)$, we also set 
$\st(Y, \sd \Y) := \bigcup_{v \in Y} \st(v, \sd \Y)$ for $Y \subseteq
V(\Y)$.

\begin{proposition}
\label{p:join_nice}
  Let $\X$ and $\Y$ be pure simplicial complexes, $\dim \X, \dim \Y \geq 0$,
  and $Y$ be a nonempty subset of $V(\Y)$.
  Assume that $\sd \X$ and $(\sd \Y, Y)$ are star decomposable in
  vertices and $\st(Y, \sd \Y)$ has trivial reduced homology groups. 
  Let $\prec$ be an arbitrarily
  total order on $V(\X) \sqcup V(\Y)$ satisfying: 


\begin{enumerate}[(i)]
  \item The restriction of $\prec$ to $V(\X)$ induces a star decomposition in
    vertices of
    $\sd(\X)$;
  \item The restriction of $\prec$ to $V(\Y)$ induces a star decomposition in
    vertices of
    $\sd(\Y, Y)$; and
  \item $Y = (V(\X)\sqcup V(\Y))_{\succ \hat{x}}$ where $\hat x$ is the last
    vertex of $V(\X)$ in $\prec$.
  
\end{enumerate}
  Then $\sd(\X * \Y, Y)$ is star decomposable in vertices in the order $\prec$ on $V(\X
  * \Y) = V(\X) \sqcup  V(\Y)$. 
\end{proposition}
For the proof, we need a following auxiliary lemma which will be useful
in the induction.


\begin{lemma}\label{l:homology_property}
  Let $\Y$ be a pure simplicial complex and $Y \subseteq V(\Y)$. Assume that the pair $(\sd\Y, Y)$ is star-decomposable
  in vertices in some total order $\prec$ on $V(\Y)$ and that $\st(Y,
  \sd \Y)$ has trivial reduced homology groups.
  Then $\st(V(\lk(y, \Y))_{\succ y}, \sd (\lk(y, \Y)))$
  has trivial reduced homology groups as well for all $y \in Y$ except for the last
  vertex in $Y$.
\end{lemma}

\begin{proof}
  Let $y \in Y$ be different from the last vertex in the order $\prec$. First, we
  show that $\st(Y_{\succ y}, \Y)$ has trivial reduced homology groups. 

  Since the pair $(\sd\Y, Y)$ is star decomposable in vertices,
  Theorem~\ref{t:sd_vd} implies that $\sd \Y$ is vertex decomposable. In
    addition, we get that $\sd \Y$ is vertex decomposable in a shedding order
    $\succ''$
    extending $\succ'$ where is derived from $\succ$. (We recall that the
    definition of the derived order is given above the statement of
    Lemma~\ref{l:star_induced}.)
    In particular, $\st(Y, \sd \Y)$ and later $\st(Y_{\succ y}, \sd\Y)$
    are intermediate steps in the sequence of complexes obtained by gradually
    removing vertices of $\Y$ in the given shedding order $\succ''$. 
    
    We also get that $\st(Y, \sd \Y)$ and $\st(Y_{\succ y}, \sd\Y)$ are
    shellable by~\cite{provan-billera80} (see Theorem~2.8 and the note below
    Definition~2.1 in~\cite{provan-billera80}). Therefore, each of them is homotopy equivalent to a wedge of $d$-spheres
    where $d = \dim \Y$; see \cite[Theorem 12.3]{kozlov08}. Since
    $\st(Y, \sd \Y)$ has trivial homology groups, this must be a trivial wedge.
    However, following the shedding order from $\st(Y, \sd \Y)$ to
    $\st(Y_{\succ y}, \sd\Y)$, we cannot introduce homology in dimension $d$
    when gradually removing vertices. Therefore, $\st(Y_{\succ y}, \sd\Y)$ has
    to be homotopy equivalent to a trivial wedge as well showing that
    $\st(Y_{\succ y}, \sd \Y)$ has trivial reduced homology groups.

    Note that $\st(Y_{\succeq y}, \sd \Y)$ has trivial reduced homology groups
    as well by analogous reasoning.

  Now, by Lemma~\ref{l:links_in_link},
  \begin{align*}
   \st(V(\lk(y, \Y))_{\succ y}, \sd (\lk(y, \Y))) \cong \O_{\sd \Y}(y, V(\Y)_{\succ y}).
  \end{align*}
 

  We use a Mayer-Vietoris sequence for $\st(Y_{\succeq y},\sd \Y)$ covered by
  $\st(y, \sd(\Y))$ and  $\st(Y_{\succ y}, \sd \Y)$. Then $\st(y, \sd(\Y)) \cap \st(Y_{\succ y}, \sd \Y) =
  \O_{\sd \Y}(y, Y_{\succ y}) = \O_{\sd \Y}(y, V(\Y)_{\succ y})$ and
  we get the following long exact sequence
  \begin{align*}
    \cdots &\longrightarrow \tilde{H}_{n+1}(\st(Y_{\succeq y},\sd \Y)) \longrightarrow \tilde{H}_n(\O_{\sd \Y}(y, V(\Y)_{\succ y})) \longrightarrow \\
    &\longrightarrow \tilde{H}_n(\st(y, \sd \Y)) \oplus \tilde{H}_n(\st(Y_{\succ y}, \sd \Y)) \longrightarrow  \tilde{H}_{n}(\st(Y_{\succeq y},\sd \Y))
    \longrightarrow \cdots
  \end{align*}
  All $\st(Y_{\succeq y},\sd \Y)$, $\st(y, \sd\Y)$ and  $\st(Y_{\succ y}, \sd \Y)$ have trivial reduced homology groups.
  Therefore, $\tilde{H}_n(\O_{\sd \Y}(y, V(\Y)_{\succ y})) \cong \tilde{H}_n(\st(V(\lk(y, \Y))_{\succ y}, \sd (\lk(y, \Y)))$ is trivial for all $n \in \Z$.
\end{proof}

\begin{proof}[Proof of Proposition~\ref{p:join_nice}]
  
Similarly, as in the proof of Proposition~\ref{p:join}, we proceed by induction on $\dim( \X \ast \Y)$.

First, we observe that the case $\dim \Y = 0$ is covered by
Proposition~\ref{p:join}. Indeed, the only issue is to verify (iii) of
Proposition~\ref{p:join}. If $\dim \Y = 0$, then $Y$ must contain a single
vertex $\hat y$ (due to the condition on homology of $\st(Y, \sd(\Y))$). Consequently,
(iii) (of this proposition) implies that the last two vertices of $\prec$ are $\hat x$ and $\hat
y$ which verifies (iii) of Proposition~\ref{p:join}.


%
 
\medskip

  Now, let us assume $\dim \X \geq 0$ and $\dim \Y \geq 1$. The order condition is satisfied since
  $Y$ is non-empty and it is equal to $(V(\X)\sqcup V(\Y))_{\succ \hat{x}}$ by
  (iii).
  
\medskip

For checking the link condition, we consider arbitrary $z \in V(\X) \sqcup
V(\Y)$ distinct from the last vertex. We need to check star decomposability of the pair
\begin{equation}
\label{e:z_lc}
(\sd(\lk(z, \X * \Y)), V(\lk(z, \X * \Y))_{\succ z}).
\end{equation}

If $z \in V(\X) \setminus \{\hat{x}\} \sqcup V(\Y) \setminus Y$, then the analysis is the same as in the proof of Proposition~\ref{p:join}.
  
If $z = \hat{x}$, the pair~\eqref{e:z_lc} becomes
$$ 
(\sd(\lk(\hat{x}, \X)\ast \Y), Y).
$$
If $\dim \X = 0$, then we further get $(\sd \Y, Y)$ which is star decomposable
in vertices by the assumptions. If $\dim \X \geq 1$, then $\dim \lk(\hat{x}, \X)
\geq 0$ and we can use the induction (note that $\sd \lk(\hat{x}, \X)$ is star
decomposable in vertices by the last vertex condition for decomposition of $\sd
\X$).

  
Finally, by assuming $z \in Y \setminus \{\hat{y}\}$, where $\hat{y}$ is the last vertex of $\prec$, we get the pair
  \begin{equation}
(\sd(\X \ast \lk(z, \Y)), V(\lk(z, \Y)_{\succ z}).  
    \label{al:link}
  \end{equation}
  By Lemma~\ref{l:homology_property} $\st(V(\lk(z, \Y)_{\succ z}, \sd (\lk(z, \Y)))$ has trivial reduced homology groups. 
  Therefore, \eqref{al:link} is star-decomposable in vertices by the induction
  hypothesis. (Here we use that $\dim \lk(z, \Y) \geq 0$ and that $(\sd \lk(z,
  \Y), V(\lk(z, \Y)_{\succ z}))$ is star decomposable in vertices by the link
  condition for the decomposition of $(\sd \Y, Y)$.) 

\medskip

Finally, we check the last vertex condition. We need star decomposability in
vertices of 
$
\sd \lk(\hat{y}, \X \ast \Y).
$
Note that $\lk(\hat{y}, \X \ast \Y) = \X \ast \lk(\hat{y}, \Y)$ as both
  sides contain simplices of the form $\xi \cup \eta$, where $\xi \in \X$,
  $\eta \cup \{\hat y\} \in \Y$, and $\hat y \not\in \eta$. Thus, we need star
  decomposability in vertices of $\sd(\X \ast \lk(\hat{y}, \Y))$.
This is star decomposable in vertices by Proposition~\ref{p:join}. (Here, we
again use that $\dim \lk(\hat y, \Y) \geq 0$ and also that $\sd \lk(\hat y,
\Y)$ is star decomposable in vertices by the last vertex condition in
the decomposition of $\sd \Y$.)
\end{proof}

\section{Proof of the main result}
In this section, we prove Theorem~\ref{t:star_decomposable} which also finishes
the proof of Theorem~\ref{t:main}.

We first need two auxiliary observations that we will use in the proof.

\begin{observation}
\label{o:boundary}
  The boundary of a simplex $\partial \sigma$ satisfies the (HRC) condition.
\end{observation}

\begin{proof}
We prove the observation by induction on $\dim \sigma$, starting with $\dim
\sigma = 0$, in which case $\partial \sigma = \emptyset$. 
If $\dim \sigma > 0$, let $\sigma' \subsetneq \sigma$. We need to check that $\lk (\sigma', \partial
\sigma)$ satisfies the (RC) condition. This link is again a boundary of
a simplex. If $\sigma' \neq \emptyset$, we get a simplex of small dimension,
therefore, we can use the induction. If $\sigma = \emptyset$, then $\lk
(\sigma', \partial \sigma) = \partial \sigma$ which is collapsible after
removing an arbitrary facet (it is a cone then). 
\end{proof}

\begin{observation}
  \label{o:rearrange}
  Let $\K$ be a collapsible complex and $w$ be an arbitrary vertex of $\K$. Then
$\K$ collapses to $w$.
\end{observation}

\begin{proof}
First, we use the well known fact that the collapses of $\K$ can be rearranged
  so that they are ordered by non-increasing dimension~\cite[Section 3]{whitehead39}.
  In particular, this means that $\K$ collapses to a graph $\G$ with $V(\G) =
  V(\K)$. This graph must be a tree as $\K$ is collapsible, and we can further
  rearrange the collapses of $\G$ so that $w$ is the last vertex.
\end{proof}

Now we prove Theorem~\ref{t:star_decomposable} by induction on the dimension of $\K$. We know that
$\K$ satisfies the (RC) condition. Therefore, there are facets $\phi_1, \dots,
\phi_t$ of $\K$ such that $\K' := \K - \{\phi_1, \dots, \phi_t\}$ is collapsible. We
further consider a sequence $(\K_1, \dots, \K_s)$ of elementary collapses of
$\K'$ where $\K' = \K_1$, $\K_s$ is a vertex (denoted by $z$), and $\K_{i+1}$
arises from $\K_i$ by removing faces $\sigma_i$ and $\tau_i$ where $\sigma_i
\subset \tau_i$ and $\dim \sigma_i = \dim \tau_i - 1$, and $\tau_i$ is the
unique maximal face containing $\sigma_i$. Then we consider the following total order
$\prec$ on nonempty faces of $\K$, that is, vertices of $\sd \K$:

$$
\phi_1 \prec \cdots \prec \phi_t \prec \sigma_1 \prec \tau_1 \prec \sigma_2
\prec \tau_2 \prec \cdots \prec \sigma_{s-1} \prec \tau_{s-1} \prec \{z\}.
$$

Our aim is to show that $\prec$ induces a star decomposition in vertices of
$\sd^2 \K$. This we will also use in the induction; that is, for
complexes $\L$ of lower dimension satisfying the (HRC) condition, we assume that a
sequence of removals of facets and collapses induces a star decomposition in
vertices of $\sd^2 \L$ as above. The proof is easy if $\dim \K = 0$ (here no
collapses are used), thus we may assume that $\dim \K > 0$ and proceed with the
second induction step.


There is essentially nothing to check for the order condition as we
provide a total order on vertices of $\sd \K$. Thus the only issue is to check
the link condition and the last vertex condition.

In order to access the vertices of $\sd \K$ more easily in the given order,
we also give them alternate names $\omega_1, \dots, \omega_k$ so that 
$$
(\phi_1, \dots, \phi_t, \sigma_1, \tau_1, \dots, \sigma_{s-1}, \tau_{s-1},
\{z\}) = (\omega_1, \dots, \omega_k)
$$
where $k= t + 2s - 1$. That is, $\phi_1 = \omega_1$, $\sigma_1 = \omega_{t+1}$,
etc.

\paragraph{Checking the last vertex condition.} Because it is easier, we check
the last vertex condition first. We need to check that $\sd \lk(\omega_k, \sd
\K)$ is star decomposable in vertices. Because $\omega_k$ is a vertex of $\K$,
this complex is isomorphic to $\sd^2 \lk (\omega_k, \K)$ by
Lemma~\ref{l:lk_sd}. Therefore, this complex is star decomposable in vertices
by induction because $\lk (\omega_k, \K)$ satisfies the (HRC) condition as this
condition is hereditary for links.

\paragraph{Checking the link condition:}
For checking the link condition, we need to check that the pair $(\sd
\lk(\omega_i, \sd \K), V(\lk(\omega_i, \sd \K))_{\succ \omega_i})$ is star
decomposable in vertices for $i \in \{1, \dots, k-1\}$. For checking this
condition we again need to `simplify' this pair so that we remove the subdivision
from the link. The tool for this is again Lemma~\ref{l:lk_sd}. For the first entry it
gives 
$$\sd \lk(\omega_i, \sd \K) \cong \sd(\sd \partial \omega_i * \sd
    \lk(\omega_i, \K)).$$ 
    We use the specific isomorphism $\Psi$ from the proof of
    Lemma~\ref{l:lk_sd} and our next task is to describe $V(\lk(\omega_i, \sd
    \K))_{\succ \omega_i})$ after applying this isomorphism.

First, we briefly describe the set $V(\lk(\omega_i, \sd \K))_{\succ \omega_i}$. The
vertices of $\lk(\omega_i, \sd \K)$ are the nonempty faces $\eta$ of $\K$ such
that either $\eta \subsetneq \omega_i$ or $\omega_i \subsetneq \eta$.
Therefore, the set $V(\lk(\omega_i, \sd \K))_{\succ \omega_i}$ consists of
faces $\eta$ as above, which in addition satisfy $\eta \succ \omega_i$. The
isomorphism $\Psi$ from the proof of Lemma~\ref{l:lk_sd} maps $\eta$ again to
$\eta$ if $\eta \subsetneq \omega_i$ and it maps $\eta$ to $\eta \setminus
\omega_i$ if $\omega_i \subsetneq \eta$. Hence
$$
\Psi(V(\lk(\omega_i, \sd \K))_{\succ \omega_i}) = V(\sd \partial
\omega_i)_{\succ \omega_i} \sqcup \{\eta \setminus \omega_i\colon \eta \supsetneq
\omega_i, \eta \succ \omega_i\},
$$
which we denote by $W$. Thus, we need to check the star decomposability in vertices of the pair
\begin{equation}
\label{e:pair_crazy}
(\sd(\sd \partial \omega_i * \sd
\lk(\omega_i, \K)), W).
\end{equation}

We distinguish several
cases according to the type of $\omega_i$.

\begin{enumerate}
  \item $\omega_i = \phi_i$, that is, $i \leq t$: 
   
    In this case, $\phi_i$ is a facet. Therefore, $\lk(\phi_i, \K) =
    \emptyset$. Also $\eta \succ \phi_i$ for all proper subfaces $\eta$.
    Therefore, the pair~\eqref{e:pair_crazy} simplifies to $(\sd(\sd \partial
    \phi_i), V(\sd \partial \phi_i))$; see Figure~\ref{f:case_phi}. 
%
    That is, we only need that $\sd(\sd \partial
        \phi_i)$ is star decomposable in vertices which follows by the induction and
    Observation~\ref{o:boundary}.

\begin{figure}
\begin{center}
  \includegraphics[page=5]{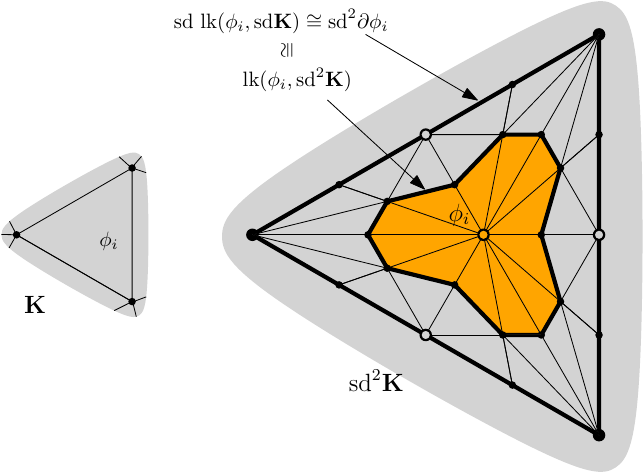}
\end{center}
\caption{Isomorphisms for verifying the link condition in case 1. We consider
  the case of the removal of the facet $\phi_i$. If we were
  checking star decomposability only, we would be interested in star
  decomposability of $\lk(\phi_i,\sd^2 \K)$. For star decomposability in
  vertices, this translates to checking the link condition for $\sd
  \lk(\phi_i,\sd \K)$ which is further isomorphic to $\sd^2 \partial \phi_i$
  (in this case, the last isomorphism is even equality).}
\label{f:case_phi}
\end{figure}
    
  \item $\omega_i = \sigma_j$ for some $j$, that is, $i > t$ and $t - i$ is
    odd:

    We need to describe $W$, for which we need to describe the faces $\eta$
    such that $\eta \subsetneq \sigma_j$ or $\sigma_j \subsetneq \eta$ such
    that $\eta \succ \sigma_j$. 
    As $\sigma_j$ induces an elementary collapse in a sequence
    of collapses of $\K'$, we get $\tau_j \succ \sigma_j$ but $\eta \prec
    \sigma_j$ for any $\eta \supsetneq \sigma_j$ different from $\tau_j$. On
    the other hand all proper subfaces of $\sigma_j$ are removed only later on
    in collapsing of $\K'$. Altogether $W = V(\sd \partial \sigma_j) \sqcup
    \{\tau_j \setminus \sigma_j\}$. See Figure~\ref{f:case_sigma} for an
    example of the pair~\eqref{e:pair_crazy} in this case.

\begin{figure}
\begin{center}
  \includegraphics[page=6]{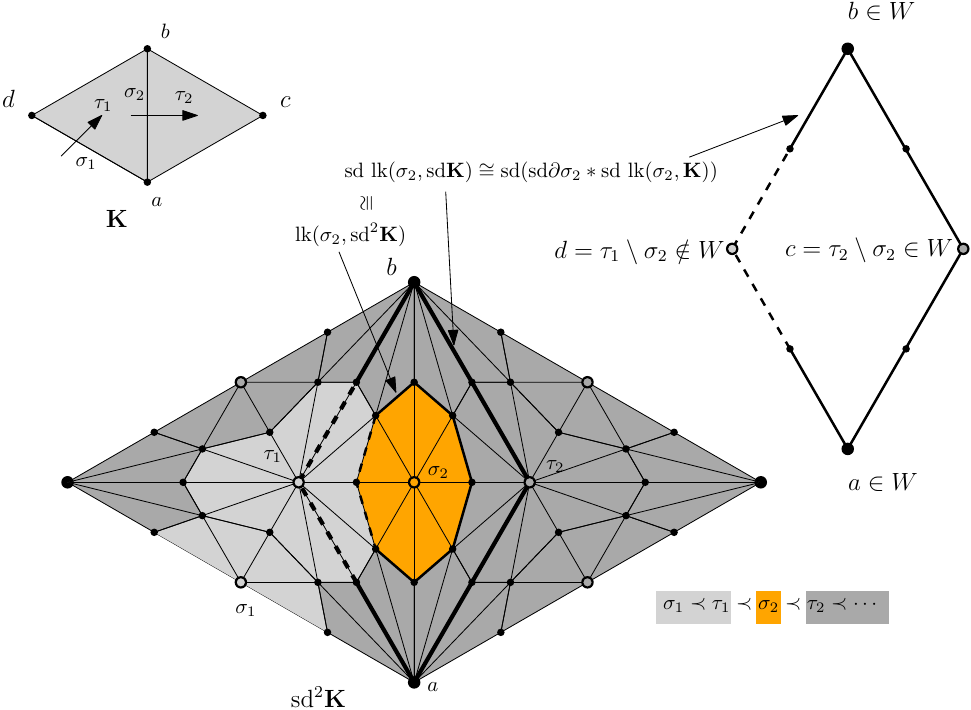}
\end{center}
  \caption{Isomorphisms for verifying the link condition in case 2. Here we
  consider the case $\sigma_j = \sigma_2$ coming from the collapses on the top
  left picture. The vertex decomposability of $(\sd
  \lk(\sigma_2, \sd \K), V(\lk(\sigma_2, \sd \K))_{\succ \sigma_2}) = (\sd
    \lk(\sigma_2, \sd \K), \{a,b, \tau_2\})$ in the  middle picture translates
    to vertex decomposability of the pair $(\sd(\sd \partial \sigma_2 * \sd
    \lk(\sigma_2, \K)), W)$ in the top right picture where $W = \{a,b,\tau_2
    \setminus \sigma_2\}$, which coincides with $V(\sd \partial \sigma_j)
    \sqcup \{\tau_j \setminus \sigma_j\}$ as required.
}
\label{f:case_sigma}
\end{figure}

    Now, we aim to use Corollary~\ref{c:join_pairs} with $$(\X, X) = (\sd
    \partial \sigma_j, V(\sd \partial \sigma_j))$$ and $$(\Y, Y) = (\sd
    \lk(\sigma_j, \K), \{\tau_j \setminus \sigma_j\}).$$
    The pair $(\sd \X, X)$ is star decomposable in vertices by
    Observation~\ref{o:boundary} and the induction. For checking star
    decomposability in vertices of $(\sd \Y, Y)$, we know that $\lk(\sigma_j,
    \K)$ satisfies the (HRC) condition. In particular, $\lk(\sigma_j,
        \K)$ is collapsible after removing some number of facets, and the
	subsequent collapses can be rearranged so that the vertex $\tau_j
	\setminus \sigma_j$ is the last vertex in the sequence of collapses.
	(If $\dim \lk(\sigma_j, \K) = 0$, then we instead rearrange the
	removals of the facets so that $\tau_j \setminus \sigma_j$ is the
	last.) Now, by induction, this sequence of removals of facets and
	collapses induces a star decomposition in vertices of $\sd \sd
	\lk(\sigma_j, \K)$ such that $\{\tau_j \setminus \sigma_j\}$ is the
	last vertex in this decomposition. This exactly means that $(\sd \Y,
	Y)$ is star decomposable in vertices.

  Altogether, Corollary~\ref{c:join_pairs} implies that the pair $(\sd(\X *
  \Y), X \sqcup Y)$ is star decomposable in vertices which is exactly the
  required pair~\eqref{e:pair_crazy}.
   
\item $\omega_i = \tau_j$ for some $j$, that is, $i > t$ and $t - i$ is
    even:
   
   We again first determine $W$. For each $\eta \supsetneq \tau_j$, we get
   $\eta \prec \tau_j$ as $\tau_j$ is a maximal face during the elementary
   collapse. On the other hand, for $\eta \subsetneq \tau_j$ we get $\eta \succ
   \tau_j$ unless $\eta = \sigma_j$ as all subfaces of $\tau_j$ have to be
   present at the moment of removing of $\sigma_j$, and $\tau_j$ immediately
   succeeds. Altogether, $W = (V(\partial \tau_j)\setminus \{\sigma_j\}) \sqcup
   \emptyset$. See Figure~\ref{f:case_tau} for an
    example of the pair~\eqref{e:pair_crazy} in this case.

\begin{figure}
\begin{center}
  \includegraphics[page=7]{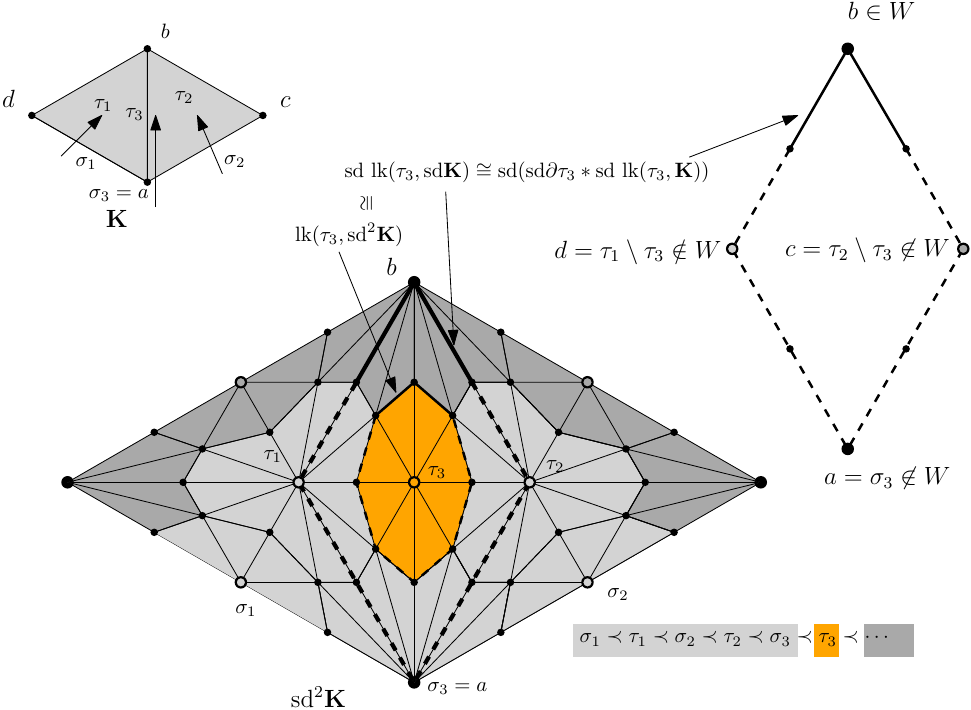}
\end{center}
  \caption{Isomorphisms for verifying the link condition in case 3. Here we
  consider the case $\tau_j = \tau_3$ coming from the collapses on the top
  left picture. The vertex decomposability of $(\sd
  \lk(\tau_3, \sd \K), V(\lk(\tau_3, \sd \K))_{\succ \tau_3}) = (\sd
    \lk(\tau_3, \sd \K), \{b\})$ in the middle picture translates
    to vertex decomposability of the pair $(\sd(\sd \partial \tau_3 * \sd
    \lk(\tau_3, \K)), W)$ in the top right picture where $W = \{b\}$, which
    coincides with $V(\sd \partial \tau_j) \setminus \{\sigma_j\}$ as required.
}
\label{f:case_tau}
\end{figure}

    We aim to use Proposition~\ref{p:join_nice} with $\X = \sd \lk (\tau_j,
    \K)$, $\Y = \sd \partial \tau_j$ and $Y = V(\sd \partial \tau_j)\setminus \{\sigma_j\}$. 
    We get that $\X$ is star decomposable in vertices by
    induction as $\lk (\tau_j, \K)$ satisfies the (HRC) condition. We also need
    that $(\sd \Y, Y)$ is star decomposable in vertices. For this we use
    Observation~\ref{o:boundary} and the induction while choosing $\sigma_j$
    to be the first face removed from $V(\sd \partial \tau_j)$. Then $Y = V(\sd
    \partial \tau_j)_{\succ' \{\sigma_j\}}$ where $\succ'$ is the corresponding
    order on $V(\sd \partial \tau_j)$. Altogether, for application of
    Proposition~\ref{p:join_nice} we choose the order $\succ'$ on $V(\sd \lk
    (\tau_j, \K)) \sqcup V(\sd \partial \tau_j)$ so that it starts with
    $\sigma_j$, it continues on $V(\sd \lk (\tau_j, \K)$ in order of a
    star decomposition in vertices of $\sd \X$ and finally it continues on
    $Y = V(\sd \partial \tau_j) \setminus \{\sigma_j\}$ in the already prescribed
    order $\succ'$. Then we get the required conclusion that $(\sd(\X * \Y),
    Y)$, which is the
    pair~\eqref{e:pair_crazy}, is star decomposable in vertices. This finishes the proof of
    Theorem~\ref{t:star_decomposable}. \hfill\qedsymbol
\end{enumerate}

\section*{Acknowledgments}
We thank to two anonymous referees for providing very useful comments.

\bibliographystyle{alpha}
\bibliography{vd}

\end{document}